\documentclass[10pt,a4paper]{article}
\usepackage[a4paper, textheight=24truecm, textwidth=16truecm]{geometry}
\usepackage{amsmath,calligra,mathrsfs,amssymb,amsthm,amsfonts,amscd,euscript,verbatim,graphicx}
\usepackage{tikz-cd}
\usepackage[english]{babel}
\usepackage{hyperref}
\usepackage{bm}
\usepackage{newlfont}
 \theoremstyle{plain}
\newtheorem{theo}{Theorem}[section]

\newtheorem{theore}{Theorem}[section]
\newtheorem{rrema}[theore]{Remark}
\newtheorem{pr}[theo]{Proposition}

 \newtheorem{lem}[theo]{Lemma}
 
 \newtheorem{coro}[theo]{Corollary}
  
\theoremstyle{remark}
\newtheorem{rema}[theo]{Remark}

\theoremstyle{definition}
\newtheorem{defi}[theo]{Definition}

\newcommand\ob{^{-1}}

\newcommand\obj{\operatorname{Obj}}
\newcommand\mo{\operatorname{Mor}}
\newcommand\id{\operatorname{id}}

\newcommand\z{{\mathbb{Z}}}
\newcommand\q{{\mathbb{Q}}}
\newcommand\af{\mathbb{A}}
 
\newcommand\p{\mathbb{P}}

\newcommand\pt{pt}
\newcommand\eps{\varepsilon}

\newcommand\al{\alpha}

\newcommand\gam{\gamma}

\newcommand\ns{\{0\}}

\DeclareMathOperator{\RHom}{\underline{\operatorname{Hom}}}

\DeclareMathOperator\inli{\varinjlim}

\newcommand\spe{\operatorname{Spec}}

\newcommand\card{Card}

 \DeclareMathOperator\Ker{\operatorname{Ker}}

\DeclareMathOperator\cha{\operatorname{char}}

\newcommand\zl{\z_l}
\newcommand\znz{\z/n\z}
\newcommand\zmz{\z/m\z}
\newcommand\zlz{\z/\ell\z}
\newcommand\zlnz{\z/\ell^n\z}
\newcommand\zlmz{\z/\ell^m\z}
\newcommand\zlmiz{\z/\ell^{m+1}\z}
\newcommand\ztz{\z/2\z}
\newcommand\hd{hd}

\newcommand\com{{\mathbb{C}}}

\newcommand\modd{\operatorname{mod}}

\newcommand\gi{\mathfrak{i}}
\newcommand\gj{\mathfrak{j}}
\newcommand\gff{\mathcal{F}}

\newcommand\ds{D^b\,Sh^{et}_{\znz-\modd}}
\newcommand\dsc{D^b_c\,Sh^{et}_{\znz-\modd}}
\newcommand\dsn{D_n}

\newcommand\iso{\xrightarrow{
   \,\smash{\raisebox{-0.5ex}{\ensuremath{\scriptstyle\sim}}}\,}}

\begin{document}

 \title{ Some general \'etale Weak Lefschetz-type theorems 
 }
 \author{Sergei I. Arkhipov,  Mikhail V. Bondarko
   \thanks{  The research was supported by the Leader (Leading scientist Math) grant no. 22-7-1-13-1 of  the Theoretical Physics and Mathematics Advancement Foundation «BASIS». }} 
 \maketitle

\begin{abstract}

We prove an 'algebraic thick hyperplane section'  weak Lefschetz-type 
 statement for the  \'etale cohomology of a  
  variety $X' / K$ that is not necessarily proper; 
  a 
  certain related theorem 
   for 
$X'/ \com\p^N$ was previously established by Goresky and MacPherson using stratified Morse theory. In contrast to their 
statement (whose formulation mentions a small neighbourhood of a hyperplane section of $X'$)  
 we  
 easily formulate and prove our 
 theorem using purely sheaf-theoretic methods; this makes it independent of the base field characteristic.  
 Convenient sheaf-theoretic methods (in particular,  proper and smooth base change) enable us to use 
 our 
  result as a substitute for the one of Goresky and MacPherson.  
 Applying the corresponding weak Lefschetz-type isomorphism along  with an argument proposed by Deligne (and implemented by W. Fulton and R. Lazarsfled in the complex case) we obtain 
 vast generalizations of important theorems of
W. Barth and A.J. Sommese 
to base fields of arbitrary  characteristics. 
  So, we compute the lower \'etale cohomology of  closed subvarieties of $\p^N$ of small codimensions and of their preimages with respect to proper morphisms (that are not necessarily finite; this statement is completely new), and also of the zero loci of sections of ample vector bundles.
Moreover, we study certain numerical characteristics of varieties that measure their failures to be local complete intersections, and relate these characteristics to our main results.



\end{abstract}

\tableofcontents

 \section*{Introduction}

In contemporary algebraic geometry, there are 
 several 'descendants' of the classical weak Lefschetz theorem (a rich collection of those that includes the theorem of Barth on the cohomology of closed subvarieties of $\p^N$ of low codimensions, can be found in \cite{lazar}). Whereas for the 
 usual  weak Lefschetz theorem the classical Artin's vanishing yields a purely algebraic proof (that works over an arbitrary characteristic base field $K$),  the proofs of several other results 
  often rely on certain topological arguments or on Hodge theory (or  De Rham cohomology). These proof methods restrict the statements obtained to the 
 case $\cha K=0$ (the literature for this case is so vast that the authors will not attempt to mention all of these results).

In particular, one of the powerful tools for studying these weak Lefschetz-type questions for $K=\com$  (as shown in particular in \S9 of \cite{fulaz};\footnote{The  method used in loc. cit. was proposed by Deligne.} cf. also \S3.5.B of \cite{lazar}) is 
the theorem formulated in \S II.1.2 of \cite{gorma}. For a quasi-finite morphism $s:X'\to \com\p^N$ 
 and any small enough $\varepsilon>0$ it states that the lower homotopy groups of $X'$ are isomorphic to those of $s\ob (\com\p^{N-k}_\varepsilon)$, where  $\com\p^{N-k}_\varepsilon$ is the $\varepsilon$-neighbourhood (in the sense of some Riemannian metric for $\com\p^N$) for the standard embedding $\gi$ of $\com\p^{N-k}$ into $\com\p^N$ ($0<k<N$; cf. the caution below). 
 Thus one may say that loc. cit. treats thick (multiple) hyperplane sections.  Note that even the formulation of this statement requires some 'topology', whereas the proof given in ibid. heavily relies on stratified Morse theory. Thus, both the formulation and the proof of loc. cit. are far from being algebraic; therefore, there is no chance to extend them to the 
 case $\cha K>0$. 

The current paper grew out of the following simple observation: one can easily prove a certain algebraic analogue of the aforementioned statement that 
 computes the lower (hyper)cohomology of $\com\p^{N-k}$ with coefficients in $R\gi^*Rs_{*}\znz_{X'}$ (where $n>0$ is prime to $p=\cha K$).
The corresponding statement (see Theorem  \ref{tgara})  is valid both for singular and for \'etale cohomology. 
Although its proof relies on certain results of \cite{bbd} 
 (since the perverse $t$-structure is a very convenient tool for our purposes), its base is just the classical Artin's vanishing theorem. 
 Respectively, our theorem 
  does not depend on the choice of  
 $K$ (and on its characteristic $p$); this is one of its main advantages over the corresponding statements of \cite{gorma}. 
 Also, we do not demand $s$ to be quasi-finite. 
 In addition, our sheaf-theoretic result can be easily combined with  proper and smooth base change; see (the proof of) Theorem \ref{tdef} 
  below. This allows us to use it as a substitute for the theorem of \cite{gorma} in the argument of Deligne described in \S9 of \cite{fulaz}. As a consequence, we
obtain certain cohomological analogues of the statements of loc. cit.  over fields of arbitrary characteristic (that doesn't divide $n$).  
In particular, we easily extend 
 and generalize important theorems of W. Barth to this context (cf. Theorem \ref{tbar}). 
  Moreover, we generalize Sommese's theorem, which compares the lower cohomology of a projective variety with that of the zero locus of a section of an ample vector bundle over it (see Theorem {tso}). 
  
 Now 
  let us formulate some 
 of 
  our results in the case of local complete intersection varieties. Note that the corresponding Theorems \ref{tbarth}(1,3) and 
 \ref{tsom} below are 
  more general (and recall that $n$ is coprime to $p=\cha K$).

 \begin{theo}\label{tbar}
      Let $t:T\to \p^N$ be a closed embedding, where $T$ is a connected local complete intersection of dimension $d$, $v:V\to \p^N$ be a proper morphism of relative dimension $\le d_v$, where $V$ is a connected local complete intersection of dimension $k$. Then for the morphism $u: v\ob (T)\to V$  we have the following: $H^i(-, \znz)(u)$ is bijective for any $i < w$ and is injective for $i = w$, where $w = \min (d+k-d_v-N, 2d-N+1)$. 
\end{theo}
 The proof of this theorem relies on a new statement on  the 
  cohomology of the preimages of the diagonal in $(\p^N)^q$ with respect to proper morphisms.

\begin{theo}\label{tso} Assume that $X$ is a connected projective local complete intersection of dimension $d$, $E$ is an ample vector bundle of rank $r$ on $X$, $s$ is a 
 section of $E$, $Z$ is the zero locus of $s$. Then 
$H^j(-, \znz)(\gi)$ is bijective for any $j < d-r$ and is injective for $j = d-r$,
where $\gi$ is the closed embedding $Z \to X$.
\end{theo}

Moreover, we establish the $\zl$-adic versions of these statements 
 and their singular (co)homology analogues in \S\ref{sladic}. Note also that in the case $r=1$ Theorem \ref{tso} gives the ordinary 
 weak Lefschetz theorem (for local complete intersections). 

 In the opinion of the authors, the proof of our weak Lefschetz-type Theorem \ref{tgara} below is much simpler than the corresponding argument of Goresky and MacPherson (although it is somewhat technical, it is quite easy to understand for a person acquainted with \'etale cohomology). 
   Moreover, using our basic theorem as a substitute for the 
results of \cite{gorma} appears to be a completely new (and important) idea that has several nice consequences. Some of them were beyond the reach of previously existing methods, especially in the case $p>0$; 
see 
Theorem \ref{tbarth}, Corollary \ref{hehe}, and Theorem \ref{tsom}. 
 The authors hope that our methods will become a useful tool for studying various weak Lefschetz-type questions (at least, the higher degree ones).
 
 Another important ingredient of our arguments is the perverse $t$-structure (on the derived category of $\znz$-module \'etale sheaves).  In particular, we relate it to the number of 'extra' 
  equations that are needed to characterize the variety locally (say, in 
 a projective space; we call this number {\it equational excess}). 
 Some of our results on these 
 matters may be (new and) rather interesting for themselves.

  \begin{rrema}
1. 
The relation of Theorem \ref{tgara} to \S II.1.2 of \cite{gorma} is discussed in \S \ref{i2p} below.

The proof of Theorem \ref{tgara} 
has nothing to do 
  with the arguments of 
   ibid. Yet   we should note that somewhat similar methods were previously used by Beilinson in order to establish a weak Lefschetz-type statement for general hyperplane sections of a smooth $X'\subset \p^N$ (see Lemma 3.3 of \cite{bei87} and Remark 
  4.2(3) of \cite{blefl}); the authors have also benefited from an argument of D. Arapura  and P. Sastry. 
  
  2. It is worth noting here that the perverse $t$-structure was related to weak Lefschetz-type statements in \cite{hammle1} (cf. Remark \ref{rhamm} below), and 
  that paper is closely related to certain definitions and conjectures proposed by Grothendieck. Yet, it appears to be really difficult to apply the arguments and formulations of ibid. as well as of \cite{hammle2}  in the case 
  $p>0$. 

3. 
 On the other hand, G. Lyubeznik has 
  established some interesting theorems  
 on the {\bf 'etale}  cohomology of closed  subvarieties of $\p^N$ (over base fields of arbitrary characteristics). His Theorem 10.5 of \cite{lyubez} 
  contains several Barth-type statements that 
 do not follow from our Theorem \ref{tbarth}(1,2); see 
  \S\ref{srem}.\ref{irlyub}--\ref{irlyub2} for the detail. 
 Yet he did not use the perverse $t$-structure, and it seems that the methods of ibid.
 cannot 
  be used to study the cohomology of the preimages of subvarieties in $\p^N$ with respect to  proper morphisms. 
\end{rrema}

Now we describe the contents of the paper. 

The  basic result of the paper is Theorem \ref{tgara} (our "thick hyperplane section weak Lefschetz"). 
We apply 
 it to certain $G_m^s$-bundles (constructed by Deligne) in 
  Proposition \ref{ptdiagonal}. 
This yields a certain
Barth-type theorem (see Theorem \ref{tbarth}; possibly, this is the most interesting result of the paper). 
Another important statement is Theorem \ref{tsom}.

In \S\ref{sprel} we recall some basics on the derived categories of (constructible) $\znz$-sheaves, on functors between them, and on the perverse $t$-structure. 

In \S\ref{schar} we discuss certain numbers that characterize (singularities of) varieties locally. One of them (the so-called cohomological depth) is defined in terms of the perverse $t$-structure and the other one is the aforementioned equational excess that counts "extra equations" that are needed to characterize the variety locally (say, in a projective space). 
The reader only interested in local complete intersections can ignore this theory; see Remark \ref{rlci}(2).

In \S\ref{sbase} we prove our 
   weak Lefschetz-type Theorems \ref{tdef} and \ref{tgara}. 

In \S\ref{sappl} we describe some applications of 
Theorem \ref{tgara}. Following \S9 of \cite{fulaz}, we consider certain 
 $G_m$-bundles. This  relates the cohomology of a (certain $G_m^{q-1}$-bundle over) a variety $Y$ that is proper over $(\p^N)^q$ to that of  the preimage of the diagonal. 
As a result, we extend (certain generalizations of) the theorems of Barth  to 
 base fields of arbitrary characteristic (and to $\znz$-coefficients). 
 We also prove Theorem \ref{tsom} on the cohomology of zero loci of ample line bundles, and Corollary \ref{hehe} on the cohomology of certain intersections.

In  \S\ref{sladic} we establish the $\zl$-adic versions of our main results. In the case $K=\com$ we also prove their  singular 
 homology re-formulations. The section also contains some lemmas that relate all this (co)homology to $\znz$-\'etale one. It appears that 
  all the arguments of this section are  rather standard. 

In  \S\ref{srem} we give some more remarks on the relation of our main results to that of earlier papers. 
We also prove a certain fundamental group version of Theorem \ref{tgara}(2).

In  \S\ref{sfut} we mention possible variations and applications of our result; some of them will probably be implemented in a subsequent paper.

In \S\ref{App} 
 consists of the justification of certain examples considered (earlier) 
  in our paper along with (rather long and yet standard) proofs of a few auxiliary lemmas. 

{\bf Caution}. 
If $K=\com$ then 
the obvious 
homomorpism   $h^i:H^i(X',\znz)\to   H^i(s\ob (\com  \p^{N-k}),\znz)$ between singular cohomology groups clearly factors as follows (for any $ i\ge 0$): 
 $$H^i(X', \znz)\stackrel{f^i}{\to}
\inli_{\eps\to 0}H^i(s\ob (\com\p^{N-k}_\varepsilon), \znz) \stackrel{g^i}{\to} 
H^i(s\ob (\com  \p^{N-k}),\znz).$$
 In 
 Proposition \ref{pbch}(II) below we describe an algebraic analogue of this factorization (through the \'etale cohomology group $H^i(\p^{N-k},  R\gi^*Rs_{*}\znz_{X'})$; unfortunately, the proof is somewhat formal). It turns out that this factorization simplifies the study of weak Lefschetz-type questions. In particular,  the results  of \cite[ \S II.1.2]{gorma} 
  (in the case $K=\com$) and  our Theorem \ref{tgara} (for the general case; see also \S\ref{srem}.\ref{i2p}) yield that $f^i$ is bijective for small enough $i$  (even if $s$ is not proper). 
Note that this statement does not extend to $h^i$ (in the general case); in particular,  $s\ob (\com \p^{N-k})$ can be empty. However, if $s$ is proper over 
some open $U\subset X$, $Z\subset O$,  then $g^i$ is necessarily an isomorphism for all $i$ (see Theorem \ref{tdef}(3); its 
 proof is a more or less easy combination of smooth and proper base change theorems). 
 Thus one may think about $H^*(\p^{N-k},R\gi^*Rs_{*}\znz_{X'})$ as of an 'approximation' to the cohomology of $s\ob (\p^{N-k})$ that has several nice properties. Our Theorem \ref{tgara}(1) yields that $f^i$ is an isomorphism for lower $i$ (one may say that this is a Goresky-MacPherson-type result); in Theorem \ref{tgara}(2) we use it in order to establish a ('true') weak Lefschetz-type statement, which we actively apply in \S\ref{sappl}.


The authors are deeply grateful  
to prof. M. Hoyois,    prof. G. Lyubeznik,  and  prof. I. Panin, and prof. A. Druzhinin  
 for their  interesting remarks. 

\section{Preliminaries on derived categories of sheaves and perverse $t$-structures}\label{sprel}

In this section we recall some of the properties of
the categories $\dsc(-)\subset \ds(-)$ and functors between them and introduce some notation; 
we also describe the properties of the perverse $t$-structure for $\dsc(-)$.
The statements of this section are well-known (possibly except Proposition \ref{pbch}(I.1), which is quite easy). Most of them were (essentially) proved in  SGA4, SGA4$\frac{1}{2}$, or in \cite{bbd}. 


We start with 
 some simple notation.

For an additive category $C$ by $C(A,B)$ we will denote the group $\mo_C(A,B)$.

$K$ will be our  base field of characteristic $p$ ($p$ can be zero). 
We will always assume that 
  $K$ is algebraically closed (for simplicity; 
  see \S\ref{sfut}(\ref{closure}) below). 
We will call a scheme $X$ a ($K$)-variety if it is a separated reduced $K$-scheme of finite type.

 $\pt$ is a point, $\af^N$ is the $N$-dimensional
affine space (over $K,\ N>0$), 
$\p^N$ is the projective space of dimension $N$. 

We will say that a morphism of varieties has relative dimension $\le e$ if the dimension of all of its fibres is at most $ e$.


Throughout the paper we will be interested in \'etale cohomology with $\znz$-coefficients, where $n$ is not divisible by $p$. 
 We assume $n$ to be fixed until \S\ref{sladic}. 

Following (\S0.0 of) \cite{bbd} we will always omit $R$ in $Rf_*,\  Rf_!,\ Rf^*$, and $Rf^!$  (for $f$ being a finite type morphism of varieties); moreover, we will ignore the difference between cohomology (of sheaves) and hypercohomology (of complexes of sheaves).  

\begin{defi}\label{dhn}

Let $X$ be a $K$-variety.

1. $\dsn(X)=\dsc(X)\subset
\ds(X)$ will denote the derived category of bounded complexes of \'etale $\znz$-module sheaves  with
constructible cohomology over $X$.




2. 
 For $C\in \obj 
 \dsn((X))$ and $i\in\z$,  we write $H^i(X,C)$ for the morphism group 
$
\dsn((X))(\znz_X,C[i])$. 
This is the 'usual' $i$-th (hyper)cohomology of $X$ with coefficients in $C$.

Moreover, we will often write $H_n^i(X)$ for $H^i(X,\znz_X)$.
\end{defi}

\begin{pr}\label{pds}

Let $f:X\to Y$ be a morphism of $K$-varieties; denote by $x,y$ the structure morphisms of $X$ and $Y$, respectively.

\begin{enumerate}



\item\label{ifun}  
The following pairs of adjoint functors
 are defined:
 $f^*: D^+(Sh^{et}(Y,\znz-\modd) \leftrightarrows 
  D^+(Sh^{et}(X,
 \znz-\modd) 
 :f_*$ and $f_!: D^+(Sh^{et}(X,\znz-\modd)  \leftrightarrows D^+(Sh^{et}(Y,  \znz-\modd):f^!$. These functors restrict to functors 
 $f^*: \dsn(Y) \leftrightarrows  
  \dsn(X):f_*$ and $f_!:   \dsn(X) \leftrightarrows \dsn(Y):f^!$.
Any of these four types of functors (when $f$ varies) respect compositions of  morphisms (i.e.,  compositions of morphisms are sent into the corresponding compositions of functors). 

\item\label{ik} $
\dsn (\spe K,\znz-\modd))$ is isomorphic to the bounded derived category $ D^b(\znz-\modd)$; 
this is an isomorphism of tensor triangulated categories. 

\item\label{icohh}
For any $C\in \obj \dsn(Y)$ we have
 $H^*(Y,C)\cong H^*(\spe K, y_*C)$. 
 Moreover,  the following composition 
  $H^*(Y,C)\to H^*(X,f^*C)\cong H^*(\spe K, x_*f^*C)= H^*(\spe K, y_*f_*f^*C) \cong H^*(Y, f_*f^*C)$ equals the morphism coming from 
 the adjunction  $f^*: \dsn(Y)  \leftrightarrows  \dsn(X):f_*$.

\item\label{iupstar}  
$f^*(\znz_Y)=\znz_X$. 

\item \label{ipur}
$f_*= f_!$ if $f$ is proper;
 if $f$ is an open immersion (or \'etale), then $f^!=f^*$.

\item\label{ikun} 
For 
 $Z=X\times Y$  we have the K\"unneth formula for the cohomology of $Z$, that is, 
for the  structure morphism $z:Z\to \spe K$ we have
$z_*\znz_Z\cong x_*\znz_X \otimes y_*\znz_Y$. 

\item\label{iglu}
If $i:Z\to X$ is a closed immersion, $U=X\setminus Z$, $j:U\to X$ is the complementary open immersion, $C\in \obj D^b(Sh^{et}(X,\znz-\modd))$, then
the pair of morphisms
\begin{equation}\label{eglu}
j_!j^*C \to C
\to i_*i^*C
\end{equation}
 (that comes from the adjunctions of assertion \ref{ifun})
 can be completed to a
distinguished triangle. 

\item\label{ile} If $p: E \to X$ is an affine space bundle then $id_{\dsn(X)} \cong p_{!}p^{!}$.

\end{enumerate}
\end{pr}
\begin{proof} The statements are well-known (see  SGA4 and SGA4 1/2 for the proofs) and most of them were (actively) used in \cite{bbd} (see also Theorem 6.3 of \cite{eke}). 
We will only give a few 
(more precise) references.


Assertion \ref{icohh} can be deduced from the 'classical' properties of cohomology by applying adjunctions.

Assertion \ref{ikun} 
follows easily from Corollary 1.11 of \cite{sga45fin}. 

Let us prove assertion \ref{ile}. Arguing similarly to Proposition 1.1.(i) of \cite{sga5e7} one can easily prove that $id_{\dsn(X)} \cong p_{*}p^{*}$.

By Theorem 0.9 of \cite{Gaber}, there exists a dualizing complex $K_X \in \dsn(X)$, and $K_E = p^{!}K_X$ is a dualizing complex for $E$. By Proposition 1.16 of \cite{sga5e1}, $D_X p_{!} \cong (p_{*})^{op} D_E$ and $D_E p^{!} \cong (p^{*})^{op} D_{X}$, where $D_X = \RHom(-, K_X): \dsn(X) \to \dsn(X)^{op}$ and $D_E = \RHom(-, K_E)$; hence $D_X p_{!} p^{!} \cong (p_{*})^{op} (p^{*})^{op} D_X \cong D_X$. 
 Since $(D_X)^{op} D_X \cong id_{\dsn(X)}$,  
  $p_{!} p^{!} \cong id_{\dsn(X)}$ indeed.

\end{proof}

We will also recall certain properties of  base change transformations (closely following \S4 of \cite{sga417}).

\begin{defi}\label{dbch}
For a commutative square 
\begin{equation}\label{ebch} \begin{CD} 
 X' @<{\gi'}<< Z'\\
@VV{s_X}V @VV{s_Z}V\\
X @<{\gi}<< Z
\end{CD}\end{equation}
of  morphisms of varieties 
 and $C\in \obj\dsn(X')$ we call the composition of the morphisms $\gi^*s_{X*}C\to   \gi^* s_{X*}\gi'_*\gi'^*C=  \gi^*\gi_* s_{Z*} \gi'^*C       \to s_{Z*}\gi'^*C$ coming from the corresponding adjunctions the {\it base change} morphism; denote by $B(s_X,\gi)$  the corresponding  natural transformation $\gi^*s_{X*}\implies s_{Z*}\gi'^*$.
\end{defi}

  \begin{pr}\label{pbch}
  
I. In the setting of Definition \ref{dbch} the following statements are fulfilled.
  
  1. The morphism $s_{X*}C\to  s_{X*}\gi'_*\gi'^*C$ that comes from the adjunction $\gi'^*:  
   \dsn(X')\leftrightarrows  \dsn(Z'):\gi'_*$  equals the composition of $\gi_*(B(s_X,\gi)(C))$ with the morphism $s_{X*}C\to \gi_*\gi^*s_{X*}C$ that comes from the adjunction  $\gi^*: \dsn(X)  
    \leftrightarrows \dsn(Z):\gi_*$.
  
  2.  Let (\ref{ebch}) be a cartesian square. Then $B(s_X,\gi)$ is an isomorphism  either if $\gi$ is smooth 
   (this statement is called the {\it smooth base change} theorem) or if $s_X$ is proper (this is the {\it proper base change} theorem).

II. 
Base change transformations respect compositions, that is, for a commutative diagram
$$ \begin{CD} Z'@>{i_1'}>> O' @ >{i_2'}>> X'\\
@VV{s_Z}V
@VV{s_O}V @VV{s_X}V\\
Z@>{i_1}>> O @ >{i_2}>> X \end{CD}$$
we have $B(s_X,i_2\circ i_1)=B(s_O,i_1 )(i'_{2*}(-)) \circ i_1^*(B(s_X,i_2 ) ) $.

\end{pr}
\begin{proof}
I.1. We have the composition $s_{X*}C\to \gi_* \gi^*s_{X*}C\to   \gi_*\gi^* s_{X*}\gi'_*\gi'^*C=  \linebreak   \gi_* \gi^*\gi_* s_{Z*} \gi'^*C    
\to \gi_* s_{Z*}\gi'^*C =s_{X*}\gi'_* \gi'^*C$; hence it suffices to verify that the composition  $s_{X*}\gi'_* \gi'^*C=\gi_* s_{Z*} \gi'^*C
   \to\gi_* \gi^*\gi_* s_{Z*} \gi'^*C     \to \gi_* s_{Z*}\gi'^*C$ is the identity.
The latter statement is immediate from the fact the composition of transformations $\gi_*\implies \gi_* \gi^*\gi_*\implies \gi_*$
(coming from the  adjunction $\gi^*: \dsn(X) \leftrightarrows 
 \dsn(Z):\gi_*$) is the identical transformation of $\gi_*$ (this is a basic general property of adjunctions).

2. These are the derived versions of the classical base change isomorphism statements; 
see Corollary 1.2 of \cite{sga416} and Theorem 4.3.1 of \cite{sga417}, respectively. 

II. This is also a well-known statement (see \S2 of ibid.).
We should check that $B(s_X,i_2\circ i_1)$ equals
the composition $(i_2\circ i_1)^*s_{X*}\implies    (i_2\circ i_1)^*   s_{X*}i'_{2*}i'^*_2=     i^*_{1}    i^*_{2}  i_{2*}s_{O*} i'^*_2\implies$ \
            $ i^*_{1}   s_{O*} i'^*_{2} \implies   i^*_{1}   s_{O*}  i'_{1*}i'^*_1 i'^*_{2} = i^*_{1}   i_{1*}   i'^*_{2}s_{Z*}  i'^*_{1}    \implies   s_{Z*}  i'^*_{1}  i'^*_{2}; $ this is obvious. 

\end{proof}


 Now we recall some of the properties of the perverse $t$-structure $p$ for $\dsn(-)$ (that corresponds to the 
  so-called self-dual perversity function 
 $p_{1/2}$). $p$ is given by a couple 
  ($\dsn(-)^{p \le 0}, \dsn(-)^{p \ge 0}$) of subclasses of $\obj \dsn(-)$; $p$ 
  was described in \cite[\S4.0]{bbd} (see also 
  \S2.2.11--13 of ibid.; 
   cf. \S III.1 of \cite{kw}).

We 
need neither the definition of $p$ nor of a $t$-structure in this paper; we will only use some properties of $p$ that essentially originate from \cite{bbd}.

\begin{pr}\label{ppt}
Assume that  $f:X\to Y$ is a morphism of $K$-varieties of relative dimension $\le d$.

Then for the perverse $t$-structure $p$ 
 corresponding to the self-dual perversity function $p_{1/2}$ the following statements are fulfilled.
\begin{enumerate}

\item\label{ifield} If $X=\spe K$, 
 then $p$ is just the canonical $t$-structure on $\dsn(\spe K)$ (thus it is compatible with  the canonical $t$-structure for $D^b(\znz-\modd)$).
 Equivalently, $C\in  \obj\dsn(\spe K)$ belongs to $\dsn(\spe K)^{p \ge 0}$ (resp. to  $\dsn(\spe K)^{p \le 0}$) 
 if and only if 
  $H^i(\spe K,C)=0$ for any $i<0$ (resp. for $i>0$).

\item\label{itaff} Suppose that there exists an open cover $Y$ by its affine subvariety $Y_i$, such that $f^{-1} (Y_i)$ can be presented as a union of $b$ open affine subvariety for every $i$. Then $f_!(\dsn(X)^{p\ge 0})\subset  \dsn(Y)^{p\ge 1-b}$ (here for $r\in \z$ we write $\dsn(-)^{p\ge r}$  for the class 
 $\dsn(-)^{p\ge 0}[-r]$ ). 


\item\label{itexdimd} $f_*(\dsn(X)^{p\ge 0})\subset \dsn(Y)^{p\ge -d}$.


\item\label{itshr}  $f^!(\dsn(Y)^{p\ge 0}) \subset \dsn(X)^{p\ge -d}$. 

\item\label{itupstur} 
Assume that $f$ is a smooth morphism all of whose (non-empty) fibers are 
of dimension $d$. Then $f^*(\dsn(Y)^{p = 0})\subset \dsn(X)^{p = d}$.



Moreover, if all the fibres of $f$ are non-empty and geometrically connected and for $C \in \obj \dsn(Y)$ we have $f^*C \in \dsn(X)^{p \ge d}$ then $C \in \dsn(Y)^{p \ge 0}$.

\end{enumerate}

\end{pr}

\begin{proof}

If $n=l_1^{m_1}l_2^{m_2} \ldots l_s^{m_s}$ is the canonical representation of $n$, then $\dsn(X) \cong 
 \cong D_{l_1^{m_1}}(X) \times \ldots \times D_{l_s^{m_s}}(X)$.
  Consequently, it suffices to prove all the statements in question in the case $n=\ell^m$.


Now, 
 for $n=\ell^m$
most of these statements were essentially 
established in \cite{bbd}. Note here that the proofs work in the case of $\zlmz$-coefficients since they are deduced from the "usual"    properties of (torsion) \'etale sheaves (the most difficult of them being the "ordinary" Artin's vanishing); see \S4.0 of 
ibid.

\ref{ifield}. Immediate from the definition of $p$. 

\ref{itaff}. Immediate from  \S4.2.3 of \cite{bbd}.

\ref{itexdimd}, \ref{itshr}.  These statements are contained in \S4.2.4 of ibid.

\ref{itupstur}. The first part of the assertion is contained in \S4.2.4 of ibid. as well. 
The second part is given by \S4.2.5 of \cite{bbd}.





    
\end{proof}

\section{Depth and excess 
of varieties}\label{schar}

We need some quantitative characteristics of varieties. The reader interested in local complete intersections can ignore the theory discussed in this section; see Remark \ref{rlci}(2) below for the detail. The most original statement of this section is Proposition \ref{hd}(\ref{hdintersect2}).

\begin{defi}\label{dee}
Let $V$ be a $K$-variety  of dimension $a$; assume $c,d\ge 0$.

1. 
We will say that $V$ is of {\it $\znz$-cohomological depth $d$} and write $\hd(V, n) = d$ or just $\hd(V)=d$ whenever $d$ is the supremum of those $i$ such that $\znz_V\in \dsn(V)^{p\ge i}$ (set $\hd( \emptyset, n) = +\infty$).

If $\hd(V, n) \ge a$ then we say that  $V$ is {\it $\znz$-cohomologically a complete intersection} (see Remark \ref{rlci}(1)). 


2. We will 
say that the {\it equational excess} of $V$ is at most $c$ 
 whenever for any point $v\in V$  there exists $N>0$ and a Zariski open subvariety $U$ of $V$ containing $v$, such that $U$ is isomorphic to an open subvariety of a set-theoretic intersection of $N-a+c$ hypersurfaces in $\p^N$.

 If $X$ is of equational excess at most $0$ then we will also say that $X$ is of equational excess  $0$ (for obvious reasons). 
\end{defi}

\begin{rema}\label{rlci}

1. 
Since $\znz_{\spe K} \in \dsn(\spe K)^{p = 0}$ and $f^{*}(\dsn(\spe K)^{p \le 0}) \subset \dsn(V)^{p \le a}$ (see Proposition 4.2.4 of \cite{bbd}), $\hd(V, n) \le a$.

 2.  Clearly, any equidimensional  local complete intersection (in particular, any smooth variety) 
     has equational excess 
      $0$. 
     More generally, $V\subset Y$ has equational excess $0$ if $V$ is equidimensional and locally a  set-theoretic complete intersection inside an equidimensional local complete intersection variety $Y$. 
     
Consequently, any  
 $V$ of this sort 
 is $\znz$-cohomologically a complete intersection (that is,
has  $\znz$-cohomological depth $\ge \dim V$); see Proposition \ref{hd}(\ref{iolstci}) below.

 3.    On the other hand, the authors do not know whether any $V$ of equational excess $0$ is a  set-theoretic complete intersection inside a local complete intersection variety. 
         Moreover, the authors do not know of 
any general estimates for 
 $r$ such that $V$ is locally set-theoretically defined by at most $r$  equations in $X$ (cf. \cite{lyubez}). This makes a comparison of some of our results with those of ibid. rather difficult; cf. 
 \S\ref{srem}.\ref{irlyub}--\ref{irlyub2} below.

\end{rema}


Let us prove 
  some properties of our definitions. It is worth noting that 
  part  \ref{hdintersect} of the proposition below follows from the remaining assertions; its main purpose is to be a step in the proof of assertion \ref{hdintersect2}.

\begin{pr} \label{hd}

Let $X$ and $Y$ be $K$-varieties.

\begin{enumerate}



\item\label{itlocal} The following assumptions on 
$X$ and $d\ge 0$ are equivalent.

(i) $\hd(X,n)\ge d$;

(ii) for any open $U\subset X$ we have  $\hd(U,n)\ge d$;

(iii) There exists an open cover $X$ by its subvarieties $X_i$ such that  $\hd(X_i,n)\ge d$ for any $i$.

  \item \label{prod}
 $\hd(X \times Y, n) \ge \hd(X, n) + \hd(Y, n)$ (cf. Remark \ref{rhamm}(\ref{iprod}) below).

\item \label{hdintersect}
Assume that $X$ and $Y$ are subvarieties of $\af^N$, such that $Y$ is a set-theoretic intersection of $r$ hypersurfaces in $\af^N$. Then $hd(X \cap Y, n) \ge \hd(X, n) - r$.

\item\label{iolstci} 
 Assume that $X$ is of dimension $a$ 
  and has equational excess $\le c$ (see Definition \ref{dee}(2)). Then $hd(X, n) \ge a-c$.

  \item \label{hdintersect2}
  Assume that $Z$ is an equidimensional smooth variety of dimension $N$ and $X$ and $Y$ are subvarieties of $Z$. Then $hd(X \cap Y, n) \ge \hd(X, n) + hd(Y, n) - N$.

  \item \label{igmbundle}
Let $\pi: E \to X$ be a locally trivial $G_m^q$-bundle or an affine space of rank $q$. Then $\hd(E, n) = \hd(X, n) + q$.

\end{enumerate}

\end{pr}

\begin{proof}
\begin{enumerate}
\item Condition (i) implies condition (ii) according to Proposition \ref{ppt}(\ref{itshr}), whereas (ii) obviously implies (iii).

 To verify that condition (iii) implies condition (i) one should just look at the definition of $p$.

 \item Similarly to the proof of Proposition \ref{ppt}, we can apply the $\znz$-version of the results of \cite[\S4.2]{bbd} in the proof of this proposition.
 
So,  recall that the external tensor product $\znz_X \boxtimes \znz_Y$ (see \S4.2.7 of \cite{bbd}) is well known to be isomorphic to   
  $\znz_{X \times Y}$.  
   Hence our assertion easily follows from the $\zlmz$-version of Proposition 4.2.8 of \cite{bbd} (cf. Lemma \ref{l-adic} below).
   
   \item 
By $i$ denote the corresponding embedding $X \cap Y \to X$.

 Since $i_*$ is conservative and $t$-exact (see \S4.2.4 of \cite{bbd}), it suffices to verify that $i_*\znz_{X\cap Y}\linebreak \cong i_*i^*\znz_{X} \in  \dsn(X)^{p\ge \hd(X, n) - r}$.
By definition, $\znz_X \in \dsn(X)^{p\ge \hd(X, n)}$. Consider the following (rotation of the) distinguished triangle given 
by Proposition \ref{pds}(\ref{iglu}): 
$$ \znz_{X} \to 
 i_*i^*\znz_{X} \to j_!j^*\znz_{X}[1].$$ 
 
 We have $j_!j^*\znz_{X}[1]\cong j_!\znz_{U}[1]$, where $j$ is the embedding of $U = X \setminus (X\cap Y) \to X$. Next, according to  assertion \ref{itlocal} we have $\znz_{U}[1] \in \dsn(U)^{p \ge \hd(X, n)-1}$. Since $j^{-1}(V)$ is a union of $r$ open affine varieties for any open affine subvariety $V$ of $X$, hence by Proposition \ref{ppt}(\ref{itaff}) $j_! (\dsn(U)^{p \ge 0}) \subset \dsn(X)^{p \ge 1 -r}$. Consequently, $j_!\znz_{U}[1] \in  \in \dsn(X)^{p \ge \hd(X, n) - r}$; thus the same is true for $i_* \znz_X$ and we obtain the result in question.

 \item According to assertion \ref{itlocal}, we can assume that $Y$ is a set-theoretic 
intersection 
 of an open 
 subvariety $V$ of 
$\af^N$ (for some $N>0$) with $N-a+c$ hypersurfaces. 

Since $V$ is smooth, Proposition \ref{ppt}(\ref{itupstur}) combine with Proposition \ref{pds}(\ref{iupstar}) yields that $\znz_V \in \linebreak\dsn(V)^{p = N}$; hence $hd(V, n) = N$. 

 
 Now the previous assertion yields the result.

\item According to assertion \ref{itlocal} we can assume that $X, Y$ and $Z$ are subvarieties of $\af^m$ (for some $m \ge 0$).

We have $X \cap Y$ isomorphic to an intersection of $X \times Y$ and $\Delta$ in $Z \times Z$, where $\Delta \cong Z$ is the diagonal in $Z \times Z$.

Since $\Delta$ is a local complete intersection and $Z \times Z$ is smooth, $\Delta$ is a local complete intersection inside $Z \times Z$ (see Proposition 5.2 of \cite{sga214}), that is, it is locally defined by $N$ equations inside $Z \times Z$. Since $X \times Y \subset Z \times Z$, 
 assertion \ref{itlocal} 
  enables us to 
   replace $X \cap Y$ with an intersection of $X \times Y$ with $N$ hypersurfaces.

According to assertion \ref{prod} $\hd(X \times Y, n) \ge  \hd(X, n) + \hd(Y, n)$; hence assertion \ref{hdintersect} yields the result.

\item 



It follows from Proposition \ref{ppt}(\ref{itupstur}) for $f = \pi$ and Proposition \ref{pds}(\ref{iupstar}).


\end{enumerate}
\end{proof}

\begin{rema} \label{rhamm}
    \begin{enumerate}
        \item \label{irhamm}
    Our definition of $\znz$-cohomological depth is closely related to the notion of rectified homological depth 
     introduced in \S1.1.3  of \cite{hammle1}; see Corollary 1.10 of ibid. However, in ibid. the (constant) sheaf $\q_V$ was considered (instead of $\znz_V$). What it even more important, ibid. and related papers (including \cite{hammle2}) treat the complex analytic setting only; 
     it appears that their results and arguments cannot be carried over to the case $p>0$. 

Recall also that the notion of rectified homological depth is closely related to the so-called rectified homotopical depth that was essentially
defined by Grothendieck (and is the central notion of \cite{hammle1}).

  \item\label{iexa1}
Let us demonstrate that  $hd(X, n)$ can depend on $n$ (if we fix $X$). 
Our example also demonstrates 
 that there exists a connected $X_m^N$ of dimension $N$ 
  such that $hd(X_m^N, n)=N$ for some $n>1$ but 
   the equational excess of $X$ is positive.

For simplicity,  we take $K = \com$. We choose $N\ge2$ and take the following action of the group $\zmz$ on $\com^N$: $g^k ((x_1, x_2, \ldots x_N))= (\zeta_m^k \cdot x_1, \zeta_m^k \cdot x_2, \ldots \zeta_m^k \cdot x_N)$ for all $k \in \z$ and
$(x_1, x_2, \ldots x_N)\in \com^N$; here $g$ is a generator of  $\zmz$  and $\zeta_m$  is a primitive $m$th root of unity. We set  $X^N_m$ to be the quotient variety $\com^N / \zmz$ by this action. So, we claim that $\hd(X^N_m, n) = N$ if $n$ and $m$ are coprime and $\hd(X^N_m, n) = 2$ otherwise. 
The proof will be given in \S \ref{App} below.

 \item\label{iprod} It appears that $\hd(X \times Y, n) = \hd(X, n) + \hd(Y, n)$ whenever $n$ is a prime power. Yet we did not check the details and will not use this statement below. 
 
 On the other hand, 
 if $n=n_1n_2$, where the $n_i$ are coprime and bigger than $1$, $N>2$, then for the varieties $X^N_{n_i}$ described above one can easily check the following statements: $\hd(X^N_{n_i}, n)=2$ for $i=1,2$, and $\hd(X^N_{n_1}\times X^N_{n_2}, n)\ge N+2$.

 \end{enumerate}

 \end{rema}

\section{Our 
'thick weak Lefschetz' theorems}\label{sbase}


In this section we consider the following setting: $s_X:X'\to X$ is
a morphism of $K$-varieties, $X$ is complete, $\gi :Z\to X$ is a closed immersion, $U=X\setminus Z$, $U'=s_X\ob(U)$. So, we have the following  commutative diagram:
\begin{equation} \label{evar} \begin{CD} 
Z'@>{\gi'}>> X' @<{\gj'}<< U'\\
@VV{s_Z}V
@VV{s_X}V @VV{s_U}V\\
Z@>{\gi}>>X @<{\gj}<< U\\
\end{CD}
\end{equation}
whose corner squares are cartesian.
 Denote the corresponding structure morphisms (whose target is $\spe K$) by $u,x,$ and $z$, 
 respectively.

It will be convenient for us to use the following definitions related to cohomology.

\begin{defi}\label{dconn}

\begin{enumerate}

\item Let $H^*=(H^i),\ i\ge 0$, be a sequence of contravariant functors from the category of $K$-varieties into $R$-modules, where $R$ is a ring.


Then we say that a $K$-morphism $f$ is {\it $d$-connected with respect to $H^*$} whenever 
 $d\in \z$, $H^i(f)$ is bijective for any $0\le i<d$, and $H^d(F)$ is injective (if $d\ge 0)$.\footnote{So, these conditions are vacuous if $d<0$.} 

Respectively 
(cf. Definition \ref{dhn}(2)) we will just say that $f$ is  $d$-connected or  $d$-connected with respect to $H^*_n$ if $f$ is  $d$-connected with respect to the functors $H^i_n=H^i(-,\znz_X)$.

\item For $w\ge 0$ we will say that a morphism $s_U:U'\to U$ (cf. (\ref{evar})) is {\it $w$-good modulo $n$} if $u_{!} s_{U*} \znz_{U'}\in \dsn(\spe K)^{p\ge w + 1}$ (recall  that  $u:U\to \spe K$ is the corresponding structure morphism).

\item \label{gff} We will write $\gff$ for the morphism $s_{X*}\znz_{X'}\to \gi_*\gi^* s_{X*}\znz_{X'}$, where the morphism $s_{X*}\znz_{X'}\to \gi_*\gi^*s_{X*}\znz_{X'}$ comes from the adjunction  $\gi^*: \dsn(X)  \leftrightarrows \dsn(Z):\gi_*$.

\end{enumerate}
    
\end{defi}

Let us relate these definitions 
 with each other and discuss some equivalent conditions.

\begin{theo}\label{tdef}
Adopt the assumptions above.

\begin{enumerate}

\item \label{i'1} We have $u_{!} s_{U*} \znz_{U'} \cong x_{*}\gj_{!} \gj^*s_{X*}\znz_{X'}$. 
 
 Consequently, $s_U$ is $w$-good modulo $n$ if and only if the homomorphism $H^i(X,-)(\gff)$ (see Definition \ref{dconn}(\ref{gff})) 
is a bijection if $i<w$, and is an injection 
if $i=w$.


\item \label{i'2}  Assume that $p_E: E \to U$ is an affine space bundle 
 (that is, $p_E$ is Zariski-locally isomorphic to $\af^r_V \to V$ for some $r\ge 0$). Then $u_{!} \cong e_{!} p_E^{!}$.


\item \label{i'3} Suppose that there exists an open $O\subset X$ such that $Z\subset O$ and the restriction of $s_X$ to the preimage $O'$ of $O$ is proper. Then the base change transformation $B(s_X,\gi):\gi^*s_{X*} \implies s_{Z*}  \gi'^*$ is an isomorphism.

    Consequently, $s_U$ is $w$-good modulo $n$ if and only if $\gi'$ is  $w$-connected with respect to $H^*_n$.
    
    \end{enumerate}
    
\end{theo}
\begin{proof}
1. Applying (\ref{eglu}) to $s_{X*}\znz_{X'}$, we obtain  a distinguished triangle $$\gj_{!} \gj^*s_{X*}\znz_{X'}  \to
s_{X*} \znz_{X'}\to \gi_{*} \gi^*s_{X*}\znz_{X'} \to \gj_{!} \gj^*s_{X*}\znz_{X'}[1]$$ in  $\dsn(X)$.
It yields that $H^i(X,-)(\gff)$ (see Definition \ref{dconn}(\ref{gff})) 
is a bijection if $i<w$, and is an injection 
if $i=w$ if and only if $x_*\gj_{!} \gj^*s_{X*}\znz_{X'}\in \dsn(\spe K)^{p\ge w + 1}$. Thus it remains to prove that  $u_{!} s_{U*} \znz_{U'} \cong x_*\gj_{!} \gj^*s_{X*}\znz_{X'}$.

 Since $\gj^*s_{X*}\cong  s_{U*}\gj'^*$ (by smooth base change; see Proposition \ref{pbch}(I.2)),  $x_*\gj_{!} \gj^*s_{X*}\znz_{X'} \cong x_*\gj_{!} s_{U*}\gj'^* \znz_{X'}$. 
 
  Lastly, since $x$ is proper, $x_! = x_*$ according to  Proposition \ref{pds}(\ref{iupstar}) and Proposition \ref{pds}(\ref{ipur}) yields that $j'^* \znz_{X'} = \znz_{U'}$. Hence $x_*\gj_{!} s_{U*}\gj'^* \znz_{X'} = x_!\gj_{!} s_{U*} \znz_{U'}=  u_{!} s_{U*} \znz_{U'}$.


2. Immediate from  Proposition \ref{pds}(\ref{ile}).

3.   Consider the commutative diagram
\begin{equation}\label{eqo} \begin{CD}
Z'@>{i_1'}>> O' @ >{i_2'}>> X'\\
@VV{s_Z}V
@VV{s_O}V @VV{s_X}V\\
Z@>{i_1}>> O @ >{i_2}>> X \end{CD}
\end{equation}
We have $\gi^*=i^*_{1} i^*_{2}$. Proposition \ref{pbch}(II) yields that $B(s_X,\gi)$ is the composition of transformations $\gi^*s_{X*}\stackrel{B_1}{\implies} i^*_{1}   s_{O*} i'^*_{2}\stackrel{B_2}{\implies} s_{Z*}  i'^*_{1} i'^*_{2}=s_{Z*}  i'^*$. 
 Next, $B_1$ is an isomorphism by the
smooth base change theorem (for the right hand side square of (\ref{eqo})), whereas $B_2$ is an isomorphism by  the proper base change theorem (for the left hand side square of (\ref{eqo})); see   Proposition \ref{pbch}(I.2). Hence the base change transformation $B(s_X,\gi):\gi^*s_{X*} \implies s_{Z*}  \gi'^*$ is an isomorphism. 

 Now, Proposition \ref{pds}(\ref{ifield},\ref{icohh})
yields 
 that for any $i\in \z$ the homomorphism $H^i(f,\znz)$ 
  can be obtained by applying the functors $H^*(X,-)$ to the $\dsn(X)$-morphism $s_{X*}\znz_{X'}\to  s_{X*}\gi'_*\gi'^*\znz_{X'}$. Next, Proposition \ref{pbch}(I.1) gives a decomposition of this morphism into the composition 
$$ s_{X*}\znz_{X'}\stackrel{\gff}\to \gi_*\gi^* s_{X*}\znz_{X'}\stackrel{B(s_X,\gi)}{\longrightarrow} s_{X*}\gi'_*\gi'^*\znz_{X'} $$
Thus the invertibility of $B(s_X,\gi)$ implies our assertion.
\end{proof}

Now we pass to proving an explicit "goodness estimate" that is central to this paper. 




\begin{theo}\label{tgara}
 
Adopt the notation and assumptions 
preceding Definition \ref{dconn} and  assume in addition that $s_U$  has relative dimension $\le d_s$ 
 and there exists an affine space bundle $p_E: E \to U$    of rank $r$
 such that $E$ can be 
 covered by $k$ open affine subvarieties. Set  $h = r + k + d_s$. 

1. Then $s_U$ is $(\hd(U')-h)$-good modulo $n$ (see Definition \ref{dee}(1) and Definition \ref{dconn}(2)). 

2. Consequently, if there exists an open $O\subset X$ such that $Z\subset O$ and the restriction of $s_X$ to the preimage $O'$ of $O$ is proper then $\gi'$ is $(\hd(U')-h)$-connected (see Definition \ref{dconn}(1)).
\end{theo}

\begin{proof}
1. According to Theorem \ref{tdef}(\ref{i'1},\ref{i'2}), it suffices to verify that 
$e_{!} p_E^{!} s_{U*} \znz_{U'}$ belongs to  \linebreak $\dsn(\spe K)^{p \ge \hd(U') - h + 1}$.

 Now, $\znz_{U'}\in \dsn(U')^{p\ge \hd(U')}$; see Definition \ref{dee}(1). 

By Proposition \ref{ppt}(\ref{itexdimd}) the functor $s_{U*}$  sends $\dsn(U')^{p \ge 0}$ into $\dsn(U)^{p\ge -d_s}$  (since $s_U$ is of  relative dimension $\le d_s$). By Proposition \ref{ppt}(\ref{itshr}),  $p_E^{!}\dsn(U)^{p\ge 0} \subset \dsn(E)^{p\ge -r}$. Thus  $p_E^{!} s_{U*} \znz_{U'}\in  \dsn(E)^{p\ge \hd(U')-d_s-r}$.

Lastly,  $e_{!}$ sends  $\dsn(E)^{p\ge 0}$ into $\dsn(\spe K)^{p\ge -k+1}$ by Proposition \ref{ppt}(\ref{itaff}).
 Hence $e_{!} p_E^{!} s_{U*} \znz_{U'}$ belongs to  $\dsn(\spe K)^{p\ge \hd(U')-h + 1}$ indeed.

2. This is an obvious combination of assertion 1 with Theorem \ref{tdef}(\ref{i'3}).
\end{proof}

\begin{rema}\label{rthiso} 
\begin{enumerate}

\item \label{i1} 
 We will need $E$ that (can be) distinct from $U$ 
 in the proof of Theorem \ref{tsom} 
  only. 
 Respectively, the reader not interested in the latter theorem (along with  Corollary \ref{csomcom}(2)) can assume that $r=0$; thus $E=U$ and $h=k + d_s$ in the notation of Theorem \ref{tgara}. 

Alternatively, in all the statements below  one can assume that the varieties 
we consider are $\znz$-cohomologically complete intersections, 
 that is, $hd(X,n)=\dim X$. Recall here that this assumption is fulfilled for  equidimensional local complete intersections; see Remark \ref{rlci}(2).   

Neither of these additional 
 restrictions makes our main results trivial.
\item \label{i2}
 Theorem \ref{tgara}(1) was inspired by the 
"thick multiple hyperplane section" version of weak Lefschetz proved by M. Goresky and R. MacPherson (for complex varieties; see the theorem in \S II.1.2  of \cite{gorma}). 
 For $\gi$ being the embedding of $Z=\com \p^{a-k}$ into $X=\com\p^a$, a quasi-finite $s_X:X'\to X$ (that is, $d_s=0$) for a 
  variety $X'$ of equational excess $0$, and any small enough $\varepsilon>0$ loc. cit. states the following: for  
  the $\varepsilon$-neighbourhood (in the sense of some Riemannian metric for $X$) $Z_\varepsilon$ of $Z$ in $X$ the natural map $\pi_i(s_X\ob(Z_\varepsilon))\to \pi_i(X')$ is an isomorphism for $i<a-k$, and is an injection for $i=a-k$.

Our theorem is an \'etale cohomological analogue of loc. cit. Note that we do not need $s_X$ to be quasi-finite. 

We will say more about the comparison of our theorems with the corresponding statements of \cite{gorma} (in particular, in the case where $X'$ is not 
 of equational excess $0$) in \S\ref{srem}.\ref{icomp}--\ref{itdef} and in \S\ref{sfut}.\ref{iest}--\ref{impr}  below.
 
\item \label{i4}   If $s_X$ is proper then  we can take $O=X$ and deduce a rather general  weak Lefschetz-type statement. The authors were not able to find this result in literature (yet cf. Theorem 9.3 of \cite{lyubez} for the case $X'=X$, whereas in the case $K=\com$ the corresponding statement probably follows from Corollary 2.1.6 of \cite{hammle2}).  
\end{enumerate}

\end{rema}

\section{Applications: Barth and Sommese-type theorems for $\znz$-cohomology} \label{sappl}

The main goal of this section is the study of the $\znz$-cohomology of schemes that are proper over $\p^N$ (in particular, of closed subvarieties of $\p^N$).

Firstly we demonstrate (in 
 Proposition \ref{ptdiagonal}) that Theorem \ref{tgara} allows to study the cohomology of the intersection of a closed subvariety of $(\p^N)^q$ with the diagonal (as well as the cohomology of  the preimages of the diagonal with respect to proper morphisms whose target is $(\p^N)^q$). 
We proceed to prove Proposition \ref{equiv} and calculate (in Theorem \ref{tbarth}) the lower $\znz$-cohomology of (the preimages of) subvarieties of $\p^N$ (of small codimension); so we extend the seminal results of Barth and others 
 to the 
 case of arbitrary characteristic (that doesn’t divide $n$). 
Our 
 arguments are closely related to the ones of  \cite[\S9]{fulaz} (cf. also \S3.5.B of \cite{lazar}); 
 see the remarks below for more detail.
We also prove Theorem \ref{tsom} on the cohomology of zero loci of ample line bundles and Corollary \ref{hehe} on the cohomology of certain intersections.


 First we 
  recall a construction due to Deligne; it is described in more detail in \S3 of  \cite{fulaz} (in the case $q=2$).
For  $q,N>0$
we note that the natural projection $a_q:(\af^{N+1}\setminus \ns)^q\to (\p^N)^q$ factors through the quotient $V_q$ of $(\af^{N+1}\setminus \ns)^q$ by the diagonal action of the multiplicative group scheme $G_m=\af^1\setminus \ns$. 
Since $(\af^{N+1}\setminus \ns)^q\subset \af^{Nq+q}\setminus\ns$, we obtain a (locally trivial) $G_m^{q-1}$-bundle $p_q:V_q\to (\p^N)^q$, a 
 $G_m$-bundle $b_q:(\af^{N+1}\setminus \ns)^q\to V_q$ and an open immersion $v_q: V_q \to \p^{Nq+q-1}$.

For a morphism of varieties $g:Y\to (\p^N)^q$ 
we will denote by $g_q:Y_q\to V_q$ (resp. by $g'_q:Y'_q\to (\af^{N+1}\setminus \ns)^q$) the base change of $g$ along $p_q$ (resp.  along $a_q$), as illustrated by the following diagram:
$$ \begin{CD}
 Y_q'@>{}>>Y_q@>{}>>Y\\
@VV{g'_q}V @VV{g_q}V@VV{g}V \\
(\af^{N+1}\setminus \ns)^q@>{b_q}>>V_q@>{p_q}>> (\p^N)^q
\end{CD}$$

\begin{pr}\label{ptdiagonal}
Let $q>1$,
 $g:Y\to (\p^N)^q$ be a proper morphism (of varieties) 
 of relative dimension $\le d_g$. 
 Denote the diagonal of $(\af^{N+1}\setminus \ns)^q$ by $\Delta_q'$. Then 
 the 
 natural morphism $\gam': g_q'^{-1}(\Delta_q')\to Y_q'$ 
 is $(\hd(Y)-d_g-qN+N)$-connected with respect to $H^*_n$. 



\end{pr}
\begin{proof}
The embedding $\Delta_q' \to (\af^{N+1}\setminus \ns)^q$ yields (after the factorization by the 
diagonal action of $G_m$) a subvariety $\Delta_q \subset {V}_q$. 
So we obtain the following diagram:
 $$ \begin{CD}
 \Delta_q'@>{}>>(\af^{N+1}\setminus \ns)^q\\  
@VV{}V@VV{b_q}V \\
\Delta_q@>{}>>V_q@>{v_q}>>\p^{Nq+q-1}
\end{CD}$$

We will write $\gi_1$ for the corresponding embedding $\Delta_q\to  V_q$. Set  $\gam$ to be the base change of $\gi_1$ along $g_q$ as illustrated by the following diagram:   
$$ \begin{CD}
 g_q\ob(\Delta_q)@>{\gam}>>Y_q@>{}>>Y\\
@VV{}V @VV{g_q}V@VV{g}V \\
\Delta_q@>{\gi_1}>>V_q@>{p_q}>> (\p^N)^q
\end{CD}$$

 Let us demonstrate $\gam$ is $(\hd(Y)-d_g-qN+N)$-connected with respect to $H^*_n$.

Set $\gi = v_q \circ \gi_1$ and $s = v_q \circ g_q$ (note that $\gi$ is also closed embedding).


Now we apply Theorem \ref{tgara}; we take $X=\p^{Nq+q-1}$, $s_X=s$, $Z=\Delta_q$, $O={V}_q$ and the trivial affine space bundle $(E=)U\to U$ (of rank $0$).
By Proposition \ref{hd}(\ref{itlocal}), $hd(Y_q \setminus g_q^{-1}(\Delta_q), n) \ge hd(Y_q, n)$ and by Proposition \ref{hd}(\ref{igmbundle}) $hd(Y_q, n) = hd(Y, n) + q - 1$. Next, $\p^{Nq + q - 1} \setminus \Delta_q$  can be covered by $qN - N + q - 1$ open affine subvarieties, 
 hence the corresponding value of $h$ equals $q - 1 + qN - N + d_g$. 
 Thus 
 $\gam$ is $(\hd(Y)-d_g-qN+N)$-connected with respect to $H^*_n$ indeed.

To conclude the proof it remains to apply the following  Lemma (\ref{lgmbundle}) in the case $f=\gam$ (respectively,  $f'=\gam'$).


\end{proof}

\begin{lem}\label{lgmbundle} Let  $f':A'\to B'$ be a morphism of locally trivial $G_m$-bundles
over the base (morphism) $f:A\to B$; 
fix a $w\in\z$.
 Then  $f$ is 
  $w$-connected with respect to $H_n^*$ if and only if $f'$ is. 

\end{lem}

\begin{proof}
    The proof is rather simple and standard. We give it in \S\ref{App} below.
\end{proof}

\begin{pr} \label{equiv}
    
    Let $t:T\to \p^N, v:V\to \p^N$ be  proper morphisms of relative dimensions $\le d_t$ and $\le d_v$ respectively, $u: 
    V \underset{\p^N}{\times} T \to T$ is the base change of $v$ along $t$,  
     and $w \le \hd(V) + \hd(T) - d_v - d_t - N$. Then the following statements are valid.

1. If $v$ is $w$-connected with respect to $H^*_n$, then 
 $u$ is $w$-connected with respect to $H^*_n$.

2. If $u$ 
 is $w + 1$-connected with respect to $H^*_n$, then $v$ is $w + 1$-connected with respect to $H^*_n$.

\end{pr}

\begin{proof}
According to Lemma \ref{lgmbundle}, it suffices to 
 verify both connectivity statements
for $t', v'$, and $u'$ instead of $t, v$ and $u$ respectively; here $t', v'$ and $u'$ is obtained from the base change of  $t, v$, and $u$ respectively along the canonical morphism $\af^{N+1} \setminus \ns \to \p^N$.

We set $q=2$ 
and take $g = t \times v$ (in the notation of Proposition \ref{ptdiagonal}); so we obtain the following diagram:

$$ \begin{CD}
V' \underset{\af^{N+1} \setminus \{0\}}{\times} T'@>{\gam'}>>  Y_2'=T'\times V'@>{}>>Y=T\times V\\
@VV{}V@VV{g'_2=t'\times v'}V @VV{g=t\times v}V \\
\af^{N+1}\setminus \ns@>{}>>  (\af^{N+1}\setminus \ns)^2@>{a_2}>> (\p^N)^2
\end{CD}$$

For our $Y=T\times V$ we have $\hd(Y)\ge \hd(V)+\hd(T)$ (see Proposition \ref{hd}(\ref{prod})) and $d_g=d_v + d_t$ (in the notation of Proposition \ref{ptdiagonal}). Hence  Proposition \ref{ptdiagonal} yields that $\gam'$ is $(\hd_V+\hd_T-d_v-d_t-N)$-connected with respect to $H_n^*$.

Now we note that $u' = pr_T \circ  \gam'$, where $pr_T:V'\times T'\to T'$ is the projection. Hence $H^i(u')=H^i(\gam')\circ H^i(pr_T)$. Hence it suffices to verify that $v'$ is $w + 1$-connected with respect to $H^*_n$ if and only if $pr_T$ is $w + 1$-connected with respect to $H^*_n$.

We have $w \le \hd(V) + \hd(T) - d_v - d_t - N$. Since $\hd(V) - d_v \le N$ and $\hd(V) - d_v \le N$ (see Remark \ref{rlci}(1)), $w + 1 \le N + 1$. Recall that the \'etale cohomology of $\af^{N+1}\setminus\ns(K)$ vanishes in all degrees between $1$ and $2N$. Consequently, $v'$ is $w$-connected with respect to $H^*_n$ if and only if the structure morphism $V' \to \spe K$ is $w$-connected with respect to $H^*_n$. Thus it suffices to prove that the structure morphism $V' \to \spe K$ is $w$-connected with respect to $H^*_n$ if and only if $pr_T: V'\times T'\to T'$ is $w$-connected with respect to $H^*_n$. Thus is remains to apply 
 Lemma \ref{product} below.

\end{proof}

\begin{lem} \label{product}
    Let $X, Y$ be $K$-varieties, $w\in \z$, $x: X \to \spe K$ be the structure morphism (see  Definition \ref{dhn}(2) and Proposition \ref{pds}(\ref{ik})). Then $x \times id_Y: X \times Y \to Y$ is $w$-connected with respect to $H^*_n$ if and only if $x$ is.
\end{lem}
\begin{proof}
    Quite simple and standard; see \S \ref{prproduct} below.
\end{proof}

We need some more notation.

 For $k\ge0 $ and a variety $T/K$ we will write $T(k)$ 
 for the union of open subvarieties $U$ of $T$, such that $\hd(U) \ge \hd(T) + k$. By $\phi(k)$ denote the dimension of $T \setminus T(k)$ (if this scheme is empty then we set $\phi(k) = -1$).

\begin{theo} \label{tbarth}
   Let $t:T\to \p^N$ be  a closed embedding. By $\theta(k)$ denote $\hd(T) + \min(0, \hd(T) + 2k - N) - \phi(k)$. 
      Then the following statements are valid.

1. $t$ is 
$\theta(0)$-connected with respect to $H^*_n$ (that is,  $t$ is $(2\hd(T)-N+1)$-connected).

2. Moreover, $t$ is 
  $\theta(k)$-connected with respect to $H^*_n$ for any $k \ge 0$. 

3. More generally, let $v:V\to \p^N$ be a proper morphism of relative dimension $\le d_v$. Then the morphism $u: v\ob (T)\to V$ 
 is $w$-connected with respect to $H^*_n$, where  $w = \min (\underset{k\ge 0}{\sup}$ $\theta(k), \hd(V) + \hd(T) - N - d_v)$.\footnote{This statement generalizes our Theorem \ref{tbar}.}
    
\end{theo}

\begin{proof}

1. 
The assertion follows from Proposition \ref{equiv}(2) applied in the case $t = v: T \to \p^N$. Indeed,  the corresponding morphism $u: T \times_{\p^n} T \to T$ is just  $id_T$; hence it is $s$-connected for any $s>0$.

2. Choose a generic projective subspace $H \cong \p^{N - \phi(k) - 1}$ of $\p^N$. Since $T \setminus T(k)$ is of dimension $\phi(k)$, we can assume that $H \cap T = H \cap T(k)$. 

By Proposition \ref{hd}(\ref{itlocal}), $\hd(T(k)) \ge \hd(T) + k$; hence $\hd(T(k) \cap H) \ge \hd(T) + k - \phi(k) - 1$ (see Proposition \ref{hd}(\ref{hdintersect2})).

Now according to the previous assertion, $H \cap T \to H$ is $(2(\hd(T) + k - \phi(k) - 1) - (N - \phi(k) - 1) + 1)$-connected. Consequently, Proposition \ref{equiv}(2) yields the result (for $v$ that is the closed embedding of $H$ into $\p^N$).


3. According to the previous assertion, $t$ is $\underset{k\ge 0}{\sup}$ $\theta(k)$-connected with respect to $H^*_n$. Hence Proposition \ref{equiv}(1) yields the result.

 \end{proof}

\begin{rema}\label{rlyub}

 1. 
  Clearly, in Theorem \ref{tbarth}(2) it makes sense to consider $k$ to be at most $\dim(T) - \hd(T)$ only. 
  

 Moreover (looking at the proof of Proposition \ref{ptdiagonal}) one can easily note that $2\hd(T) - N + 1$ in  Theorem \ref{tbarth}(1) can be replaced by
$\hd((T \times T) \setminus \Delta) - N + 1$, where $\Delta$ is the image of the diagonal embedding $T\to T \times T$. This gives an improvement of the estimate in the case where the smallest $T \setminus T(k)$ consists of a single point.
 
 2. Note 
 that in contrast to \S9 of \cite{fulaz} we do not demand 
  $v$ to be finite in part 3 of our theorem; cf. Remark \ref{rthiso}(\ref{i2}) above.

 \end{rema}

\begin{coro} \label{hehe}
    Assume $q \ge 1$, $T_j$ are closed subvarieties of $\p^N$ for $1\le j\le q$; set $d = \underset{j}{\min}$ $\hd(T_j)$. Then $\gi$ is $w$-connected with respect to $H^*_n$, where $\gi$ is the closed immersion $\bigcap_{1\le j\le q} T_j \to \p^N$, $w = \min (2d - N + 1, \sum_{j=1}^q \hd(T_j)-(q-1)N)$.
\end{coro}

\begin{proof}
    We give a proof by induction on $q$.

    \textit{Base case:} $q = 1$.

    We apply Theorem \ref{tbarth}(1) and obtain 
      $H^i(T_1,\znz)\cong H^i(\p^N,\znz)$ for $i\le 2\hd(T_1)-N + 1$.

    \textit{Induction step.} We can assume that $\hd(T_1)$ is the largest of the $\hd(T_j)$. Take $T = T_1$ and $V=\cap_{2\le j\le q}T_j$.
    
    
    By Proposition \ref{hd}(\ref{hdintersect2}), $\hd(V) \ge \sum_{j=2}^q \hd(T_j) - (q-2)N$. Hence Theorem \ref{tbarth}(3) implies that the closed immersion $u:  V \cap T \to V$ is $\sum_{j=1}^q \hd(T_j) - (q-1)N$-connected with respect to $H_n^*$.
     



    By the induction hypothesis, the  closed immersion  $v: V \to \p^N$ is   $ \min (2d - N + 1, \sum_{j=2}^q \hd(T_j)-(q-2)N)$-connected with respect to $H_n^*$. Hence $\gi$ is 
    $w$-connected with respect to $H_n^*$ indeed.


\end{proof}

Thus we are able to extend the statement that is mentioned after Remark 9.9 of \cite{fulaz} to the case of  $K$ of arbitrary characteristic;  this assertion does not appear to follow from the results of \cite{lyubez}.
    
Now let us prove an extension of  Sommese's theorem; this statement generalizes our Theorem \ref{tso}.  Our argument closely follows the  proof of \cite[Theorem 7.1.1]{lazar} (yet note that our setting is much more general). 

\begin{theo} \label{tsom}

Assume that $X$ is a projective variety. Let $E$ be an ample vector bundle of rank $r$ over $X$ (the definition of ample vector bundle is given in 6.1.1 of \cite{lazar}), $s$ is a 
 section of $E$, $Z$ is the zero locus of $s$. Then the embedding $\gi:Z\to X$ is $(\hd(X\setminus Z)-r)$-connected with respect to $H^*_n$. 

\end{theo}

\begin{proof}

We may 
 assume  $s$ to be a non-zero section.

    Take $\pi: \textbf{P}(E) \to X$ to be the projective bundle associated with $E$. By $\pi^*s$ denote the base change of $s$ along $\pi$. Also we have the tautological quotient morphism $\pi^{*}E \to \mathcal{O}_{\textbf{P}(E)}(1)$ (see Appendix A of \cite{lazar} for the detail). So we obtain the following diagram:

\[\begin{tikzcd}
	{\mathcal{O}_{\textbf{P}(E)}} \\
	{\pi^{*}E} & {\mathcal{O}_{\textbf{P}(E)}(1)}
	\arrow["{\pi^{*}s}"', from=1-1, to=2-1]
	\arrow["{s^{*}}", from=1-1, to=2-2]
	\arrow[from=2-1, to=2-2]
\end{tikzcd}\]

Denote by $s^{*} \in \Gamma(\textbf{P}(E), \mathcal{O}_{\textbf{P}(E)}(1))$ the indicated composition, and by $Z^{*} \subset \textbf{P}(E)$ the zero locus of $s^{*}$. 

Recall that $\textbf{P}(E) = \{(x, \lambda)| x \in X, \lambda \in E_x^{*}\}$; hence $Z^{*} = \{(x, \lambda)| \lambda(s(x)) = 0\}$. 

Given $x \in X$, the equation for $Z^*$ defines a hyperplane in
$P(E(x))= \pi^{-1}(x)$ if $x \in X \setminus Z$, while if $x \in Z$ there is no condition on $\lambda$. 
 Consequently, $\pi$ restricts to a mapping $p: \textbf{P}(E) \setminus Z^{*} \to X \setminus Z$, and $p$ is an affine space bundle.

Next, $\mathcal{O}_{\textbf{P}(E)}(1)$ is an ample line bundle on $\textbf{P}(E)$ 
 since $E$ is ample; hence $\textbf{P}(E) \setminus Z^{*}$ is an affine variety.

Now we apply Theorem \ref{tgara}(2) for $X = X, Z =  Z, O = X, s_X = id_{X}$ and the affine space bundle $p: \textbf{P}(E) \setminus Z^{*} \to X \setminus Z$.  This gives the result in question. Indeed, $p$ is an affine space bundle of rank $r-1$ and $\textbf{P}(E) \setminus Z^{*}$ is an affine variety. We conclude that $\gi$ is $(\hd(X\setminus Z) - r)$-connected with respect to $H_n^*$.    
\end{proof}

Now we say a little more on the relation of our results to those of \cite[\S9]{fulaz} (where for 
$K=\com$ the homotopy analogues of all of the results of this section were proved).

\begin{rema}\label{rcfulaz}
 
1. \S9 of \cite{fulaz} relies on the results of \cite{gorma}. Respectively, 
\cite[\S9]{fulaz} has 
 four features that distinguish it from 
 the results of this section.

(i) it treats only complex varieties;

(ii) it treats local complete intersections only;

(iii) all the corresponding morphisms to $\p^N$ and $(\p^N)^q$ are required to be finite;

(iv) the results are formulated in terms of homotopy instead of cohomology.

Certainly, (i) is a serious disadvantage of \cite{fulaz}. The restriction (ii) can possibly be removed via generalizing Theorem II.1.2 of \cite{gorma}.\footnote{Probably Theorem  2.1.3 of \cite{hammle2} is fine for these purposes (in the case $K=\com$); however, the authors did not check the detail.}
 It appears to be no problem to deal with  restriction (iii) similarly to \cite{gorma}; thus one has to take into account the corresponding equation excesses.

2. Now we discuss the difference between treating cohomology and homotopy.

Certainly, cohomology is much easier to compute. On the other hand, homotopy groups have certain theoretical advantages. 
One of them is that the relation between the (higher) homotopy groups of a (locally trivial) $G_m(\com)$-bundle
$X'\to X$ with those of $X$ is much easier to use than  (\ref{efrgys}). 
 In particular, the homotopy analogue of our Proposition \ref{equiv} is Theorem 9.6 of  \cite{fulaz} that states that the corresponding relative homotopy groups are just isomorphic. 


So, it would be interesting to obtain certain homotopy analogues of the results above. The second author described some of his ideas in this direction in Remark 3.4(2) of \cite{blefl}. Moreover, one can probably 
 deduce some homotopy results from their cohomological versions arguing similarly to (the proof of) Theorem 10.6 of \cite{lyubez}.
\end{rema}

\section{$\ell$-adic and singular (co)homology versions of main results}\label{sladic}

In this section we demonstrate that our main statements have natural $\zl$-adic analogues.   In the case $K=\com$ they also yield similar properties of singular (co)homology.

It appears that all the arguments of this section are  rather standard.

 Below, $\mathrm{Tors}(A)$ and $\mathrm{Tors}_{\ell}(A)$ denote the torsion and the $\ell$-torsion subgroups of an abelian group $A$, respectively.

\subsection{$\zl$-adic versions}

Fix a prime $\ell$ distinct from $p=\cha K$.

\begin{lem}\label{l-adic}
     Assume that 
    $w$ is a positive integer, and a morphism $f: X \to Y$ 
     of $K$-varieties 
      is $w$-connected with respect to $H_\ell^*$
      . Then 
  $f$ is also $w$-connected with respect to  $H^{*}(-, \z_{\ell})$ and with respect to $H_{\ell^m}^*$ for any 
   $m>0$, and $(w+1)$-connected with respect to $\mathrm{Tors}(H^{*}(-, \z_{\ell}))$ (see Definition \ref{dconn}(1)).
\end{lem}

\begin{proof}
    See \S \ref{App}.
\end{proof}

\begin{coro}\label{czl}

Assume that $t:T\to \p^N$ is a closed embedding 
 and $hd(T,\ell)=h_T$. 

     1. 
     Let  $v:V\to \p^N$ be a proper morphism of relative dimension $\le d_v$, 
       where 
        $hd(V,\ell)=h_V$. Then the morphism $u: v\ob (T)\to V$ 
       is $w$-connected with respect to $H^{*}(-, \z_{\ell})$ and $(w+1)$-connected with respect to $\mathrm{Tors}(H^{*}(-, \z_{\ell}))$, 
        where $w =  min(h_V + h_T - N - d_v, 2h_T-N+1)$.  

     
     2. Assume that 
      $E$ is an ample vector bundle of rank $r$ over $T$, and $s$ is a 
       section of $E$, $Z$ is the zero locus of $s$. Then the embedding $\gi:Z\to T$ is $(h_T-r)$-connected with respect to $H^{*}(-, \z_{\ell})$ and $(h_T-r+1)$-connected with respect to $\mathrm{Tors}(H^{*}(-, \z_{\ell}))$. 
\end{coro}
\begin{proof}
    The statements are straightforward combinations of Lemma \ref{l-adic} with Theorems \ref{tbarth}(3) and \ref{tsom}, respectively.
\end{proof}

\begin{rema}
 
1. 
Assume that a morphism $f:X\to Y$of $K$-varieties is $w$-connected with respect to  $H^{*}(-, \z_{\ell})$ (for some $w\ge 0$).

Since $H^{*}(-, \q_{\ell}) = H^{*}(-, \z_{\ell}) \otimes \q_{\ell}$ (see the end of  \cite[\S I.12]{frki} for the definition), 
$f$ is $w$-connected with respect to  $H^{*}(-, \q_{\ell})$ as well .

Now assume that $f$ (along with $X$ and $Y$) is defined over a finite field $\mathbb{F}_q$ (where $q$ is a power of $p$) and $K$ equals the algebraic closure $\overline{\mathbb{F}_q}$ of $\mathbb{F}_q$, that is, there exists a $\mathbb{F}_q$-morphism $f_0:X_0\to Y_0$ such that $f=f_0\times_{\spe \mathbb{F}_q}\spe \overline{\mathbb{F}_q}$. Then the corresponding homomorphisms $H^i(f,\q_{\ell})$ (for $i\ge 0$) are clearly $\operatorname{Gal}(\overline{\mathbb{F}_q}/{\mathbb{F}_q})$-equivariant.

 2. Next, let us fix 
  this $f_0$ and assume that $X$ and $Y$ are proper. Then combining the latter observation  with the Grothendieck-Lefschetz trace formula (see 12.3 of \cite{milne}) we obtain some relation between the numbers of points of $X_0$ and $Y_0$ over the fields $\mathbb{F}_{q^s}$ for $s>0$.

 Finally, assume in addition that $X$ and $Y$ are smooth and equidimensional. Then it is easily seen (if one applies Poincare duality) that $$|\card(X_0(\mathbb{F}_{q^s}))-\card(Y_0(\mathbb{F}_{q^s}))q^{\dim X-\dim Y}|\le c(f_0)q^{s(\dim X-w/2)},$$ where $c(f_0)$ does not depend on $s$. 
 
 Note that one can apply this inequality in the setting 
  of Corollary \ref{czl} (if 
   the domain and the codomain of $f$ are smooth; we take $f=u$ in the setting of part 1 and $f=\gi$ for part 2). 
  Moreover, the left hand side of this inequality does not depend on $\ell$ in any way; thus one can choose $\ell\in \p\setminus\{p\}$ to make $w$ (that is determined by the corresponding $hd(V,\ell)$ and $hd(T,\ell)$) 
   as large as possible.
\end{rema}

\subsection{Singular cohomology and homology in the complex case}

Till the end of the section we assume $K=\com$.

\begin{rema} \label{finiteness}
    \begin{enumerate}
        \item 
        \label{ii1} 
        It's well known that $H^{i}_{sing}(X, R)$ and $H_{i}^{sing}(X, R)$ are finitely generated $R$-modules for any variety $X/\mathbb{C}$, $i \geq 0$, and a commutative Noetherian unital ring $R$. In particular, this statement can easily be deduced from Hironaka's resolution of singularity results.

 \item \label{ii2} There are two distinct ways to define singular cohomology groups with coefficients in $\z_{\ell}$.
 
 One may either use the usual definition of singular cohomology groups with coefficients or define it as the inverse limit of singular cohomology groups with coefficients in $\zlnz$. These two definitions 
  give isomorphic functors in our case. 
 This statement follows from the universal coefficient theorem for cohomology and the fact that cohomology groups are finitely generated (see Remark \ref{finiteness}(\ref{ii1})).

\item \label{ii3} Let $X$ be a complex variety, and $F$ be a  constructible torsion sheaf on $X$ for the \'etale topology. Then $H^q_{et}(X_{et}, F) \cong H^q_{sing} (X, F)$ (see Theorem 5.2 of \cite{art}). 
      \end{enumerate}
\end{rema}

We need a few lemmas to relate \'etale cohomology and our connectivity statements with singular (co)homology.
 These lemmas are possibly well known. Their proofs are rather straightforward and do not contain any new ideas.

\begin{lem} \label{complex}

    Let $f: X \to Y$ be a morphism of $\mathbb{C}$-varieties that is  $w$-connected with respect to $H_\ell^*$ for any prime $\ell$ (where $w$ is a positive integer).
     Then $f$ is also  $w$-connected with respect to $H^{*}_{sing}(-, \z)$ and  $(w+1)$-connected with respect to $\mathrm{Tors}(H^{*}_{sing}(-, \z))$.
     
\end{lem}

\begin{proof}
    
  First $H^i_{et} (Z, \zlnz) \iso H^i_{sing} (Z_{an}, \zlnz)$ for any prime $\ell$ and for any variety $Z/\mathbb{C}$ (see Remark \ref{finiteness}(\ref{ii3}));  hence $H^i_{et} (Z, \mathbb{Z}_\ell) \iso H^i_{sing} (Z_{an}, \mathbb{Z}_\ell)$ (see Remark \ref{finiteness}(\ref{ii2})). Applying Lemma \ref{l-adic} we obtain that 
  $f$ is $w$-connected with respect to $H^*_{sing}(-,\mathbb{Z}_\ell)$ and $(w+1)$-connected with respect to $\mathrm{Tors}(H^*_{sing}(-,\mathbb{Z}_\ell)$ for any prime $\ell$.

    Next, $H_i^{sing}(Z,\z)$ is a finitely generated $\z$-module and $H_i^{sing}(Z,\z_{\ell})$ is a finitely generated $\z_{\ell}$-module (see Remark \ref{finiteness}(\ref{ii1})). Since $\z_{\ell}$ is a flat $\z$-module,  the universal coefficient theorem for homology yields that $H_{*}^{sing}(Z, \z_{\ell}) \cong H^{sing}_{*}(Z, \z) \otimes \z_{\ell}$, hence by  the universal coefficient theorem for homology $H^{*}_{sing}(Z, \z_{\ell}) \cong H_{sing}^{*}(Z, \z) \otimes \z_{\ell}$ for any prime $\ell$ and for any variety $Z$.

    Now, let $A$ be the kernel of $H^i_{sing}(-,\z)(f)$ for $i \le w$ (resp. the cokernel of $H^i_{sing}(-,\z)(f)$ for $i < w$).  Then $A \otimes \z_{\ell} = 0$ for any prime $\ell$; thus $A = 0$ and this yields the first part of the statement.

     Lastly, $H^i(Z, \z_{\ell}) \cong H^i(Z, \z) \otimes \z_{\ell}$;  hence $\mathrm{Tors}(H^i_{sing}(Z, \z_{\ell})) \cong \mathrm{Tors}((H^i(Z, \z) \otimes \z_{\ell})) \cong \mathrm{Tors}_{\ell}(H^i(Z, \z))$ for every prime $\ell$. Thus $f$ is $(w+1)$-connected with respect to $\mathrm{Tors}_{\ell}(H^{*}(-, \z))$ for every prime $\ell$; hence  $f$ is $(w+1)$-connected with respect to $\mathrm{Tors}(H^{*}_{sing}(-, \z))$.
    
\end{proof}

In the case $K=\com$  Lemma \ref{complex} clearly enables one to re-formulate our main  results in terms of singular cohomology (along with its torsion). 
 However, the following lemma 
gives shorter re-formulations.

\begin{lem}\label{homology}
    Let $f: X \to Y$ be a morphism of $\mathbb{C}$-varieties. Then $f$ is  $w$-connected with respect to $H^{*}_{sing}(-, \z)$ and  $(w+1)$-connected with respect to $\mathrm{Tors}(H^{*}_{sing}(-, \z))$ if and only if $f$ is $w$-connected with respect to $H_{*}^{sing}(-, \z)$ (that is, $H_i^{sing}(-, \z)(f)$ is bijective for any $0 \leq i < w$ and $H_w^{sing}(-, \z)(f)$ is injective).
\end{lem}

\begin{proof}
    By the universal coefficient theorem for cohomology, there exists an exact sequence \linebreak $0 \to  \mathrm{Ext}^1(H_{i-1}(X, \z), \z) \to H^{i}(X, \z) \to Hom(H_i(X, \z), \z) \to 0$. 

   Now, recall that $\mathrm{Ext}^1(A, \z) \cong \mathrm{Tors}(A)$ for any finitely generated abelian group $A$; hence \linebreak  $\mathrm{Ext}^1(H_{i-1}(X, \z), \z) \cong \mathrm{Tors}(H_{i-1}(X, \z))$.

     Let $\mathrm{F}(A)$ denote the torsion-free part of an abelian group $A$. Since $H_{i-1}(X, \z),\  H_i(X, \z),$ and $ H^i(X, \z)$ are finitely generated 
     abelian groups (see Remark \ref{finiteness}(\ref{ii1})), $\mathrm{Tors}(H^i(X, \z)) = \mathrm{Tors}(H_{i-1}(X, \z))$ and $\mathrm{F}(H^{i}(X, \z)) = \mathrm{F}(H_i(X, \z))$; this concludes the proof. 
\end{proof}

Now we combine Lemmas \ref{homology} 
 and \ref{complex} with the results above to obtain the following statements. 

\begin{coro} \label{cgaracom}
    Assume that $X$ is a complex complete variety, $\gi: Z \to X$ is a closed immersion, $U = X \setminus Z$, $s_X: X' \to X$ is of relative dimension $\le d_s$, $U'=s_X\ob(U)$, $Z' = s_X\ob(Z)$, and there exists an affine space bundle $p_E: E \to U$ of rank $r$ such $E$ can be presented as the union of $k$ open affine subvarieties. We set $w = \hd(U') - r - k - d_s$.
Assume, moreover, that there exists an open $O\subset X$ such that $Z\subset O$ and the restriction of $s_X$ to the preimage $O'$ of $O$ is proper. Then $\gi'$ is $w$-connected with respect to 
      $H^{sing}_* (-, \z)$ (see Lemma \ref{homology}).
\end{coro}

Next, Theorem \ref{tbarth}(3) and Theorem \ref{tsom} easily imply the following statements.

\begin{coro}\label{csomcom}
     Assume that $K = \mathbb{C}$. 
     
Let $t:T\to \p^N$ be a closed embedding, where $T$ is a variety of  equational excess $0$ and dimension $d$. 

     1. 
     Let  $v:V\to \p^N$ be a proper morphism of relative dimension $\le d_v$, 
      and assume that 
      $V$ is also a variety of  equational excess $0$ and dimension 
 $k$. Then the morphism $u: v\ob (T)\to V$ 
       is  $w$-connected with respect to $H^{sing}_*$, where  $w = d - N + \min (d+1, k-d_v)$.
     
     2. Assume that 
      $E$ is an ample vector bundle of rank $r$ over $T$, and $s$ is a 
       section of $E$, $Z$ is the zero locus of $s$. Then the embedding $\gi:Z\to T$ is  $d-r$-connected with respect to $H^{sing}_*$.

\end{coro}

\begin{rema}

\begin{enumerate}

\item  Similarly, for $K=\com$ it's possible to re-formulate Theorem \ref{tbarth}(2,3), Corollary \ref{hehe}, and Theorem \ref{tsom}  in terms of singular homology  and cohomological depth; cf. Corollary \ref{czl}.  
 For this purpose, one should take the minima of the corresponding $hd(-,\ell)$ for all 
 $\ell\in\p$.

     \item Moreover, one can easily generalize Lemma  \ref{complex} as follows: if a morphism $f$ (of complex varieties) is $w$-connected with respect to $H_\ell^*$ for $\ell$ running through a non-empty set $L\subset \p$ and $S=\p\setminus L$ then 
     $f$ is also  $w$-connected with respect to $H^{*}_{sing}(-, \z[S\ob])$ and  $(w+1)$-connected with respect to $\mathrm{Tors}(H^{*}_{sing}(-, \z[S\ob]))$. This statement immediately yields the corresponding $\z[S\ob])$-singular cohomology and homology versions of Theorem \ref{tbarth}(2,3), Corollary \ref{hehe}, and Theorem \ref{tsom}. 

    These statements are actual if the corresponding values of $H(-,\ell)$ depend on $\ell$; cf. Remarks \ref{rhamm}(\ref{iexa1}) and  
     \S\ref{srem}.\ref{irlyub}.  

\item 
 If in Corollary \ref{csomcom}(2) we assume in addition that $X$ is smooth 
then we obtain the original Sommese's theorem (see Theorem 7.1.1 of \cite{lazar}).

\item 
     Lemmas \ref{homology} and \ref{complex} are actually valid for arbitrary topological spaces with finitely generated singular homology.

     
   \end{enumerate} 

\end{rema}

\section{
Some remarks and a fundamental group statement}\label{srem}

\begin{enumerate}
\item\label{irlyub}
Parts 1 and 2 of 
 Theorem \ref{tbarth} are certain Barth-type theorems closely related to (part (iv) of) Theorem 10.5 of \cite{lyubez}. Note that the methods of ibid. are quite distinct from ours. The main method for comparing the cohomology of $\p^N$ with that of $T$ in loc. cit. is to look at the \'etale cohomological dimension of $\p^N\setminus T$. It appears that this argument (that relies on the well-known Proposition 9.1 of ibid.) can only be applied to closed embeddings of projective varieties; in particular, it cannot yield Theorem \ref{tbarth}(3). 

Moreover, it is not clear how to compare the 
 estimate provided by Theorem 10.5(iv) of ibid. with our statements. Note that if $T$ is (say) equidimensional of equational excess at most $e$ and of codimension $c$ in $\p^N$ then $t$ is $N-2c-2e+1$-connected; see part 1 of our theorem and Proposition \ref\hd(\ref{iolstci}). Now, loc. cit. implies that $t$ is  $N-2c-e'+1$-connected, where $e'+c$ is the number of equations needed to define $T$ locally set-theoretically {\bf inside $\p^N$}. 
 Now, it is not clear how to compare $e'$ with $e$; see Remark \ref{rlci}(3) above.   
 
 Note also that loc. cit. actually gives an even better estimate whenever "one does not need too many equations for defining $T$ outside a subvariety of dimension $\le 1$". Now, 
  Theorem \ref{tbarth}(2) gives an improvement of part 1 if $T$ is (say) $\znz$-cohomologically a complete intersection outside a subvariety of any "small" dimension; yet we are not able to "ignore" this subvariety in the estimate.

 Lastly, one can certainly combine all of the parts of   \cite[Theorem 10.5]{lyubez} 
 with 
 Proposition \ref{equiv} (similarly to the proof of Theorem \ref{tbarth}(2)) to obtain some new statements.

 \item\label{irlyub2} Let us give a concrete example where Theorem \ref{tbarth}(1) gives a better connectivity estimate than (Theorem 10.5 of) \cite{lyubez}.

  Take $K = \com$. We claim that  arguments of \cite[\S3]{Ryd} yield a closed embedding 
 $t:T=S^2 \p^3\to \p^9$; see \S\ref{App} below for the detail. 
Moreover,  $\hd(T, n) = 6$ for odd $n$ and $\hd(T, n) = 5$ for even $n$; hence the equational excess of 
 $T$ is at least $1$. Moreover, $T \setminus \Delta$ is the maximal open subvariety of $T$ whose  equational excess is $0$, where $\Delta$ is the 'diagonal' of $S^2 \p^3$.

 Now, since $\hd(T, n) = 6$ for odd $n$, our Theorem \ref{tbarth}(1) shows that $t$ is 4-connected with respect to $H^*_n$ for these $n$.

 On the other hand, 
 our observations easily yield that one cannot take $r$ to be less than $4$ in Theorem 10.5(iv) of \cite{lyubez}  
 Hence 
 loc. cit.  only implies that $t$ is $9-4-3+1=3$-connected with respect to $H_n$ for any $n>0$. The best estimate among the remaining parts of 
   \cite[Theorem 10.5]{lyubez} is given by part (iii) (in this case); it only shows that $t$ is $[10/4]+[9/4]-1=3$-connected with respect to $H_n$.

  Consequently, in this case Theorem \ref{tbarth}(1) gives a better connectivity estimate for odd values of $n$ than Theorem 10.5  of \cite{lyubez}.

  We also conjecture that there exists 
   a closed embedding $t':T'\to \p^m$ such that Theorem \ref{tbarth}(1) yields the $u$-connectivity of $t'$ with respect to $H_n$ for some value of $n>1$, and $t'$ is not $u$-connective with respect to $H_{n'}$ for some $n'\neq n$.  

 
    
\item\label{icomp} 
If we take a quasi-finite morphism $s_X':X'\to X=\p^a$ for $X'$ that is not a  local complete intersection variety, the connectivity (and "goodness") estimate given by Theorem \ref{tgara}
may be somewhat worse than 
 the one provided by  \S II.1.2  of \cite{gorma}. The reason for that is that in loc. cit. the estimate is calculated by means of using a rather complicated formula that gives a value better than our one if the locus where $U'$ "requires too many extra equations" has small dimension. We will possibly prove that our estimate can be improved similarly in a subsequent paper; cf. \S\ref{sfut}.\ref{impr} below.

  \item \label{i2p} 
      Theorem \ref{tgara} implies that the cohomological analogue of the aforementioned result of \cite{gorma}  is fulfilled "in the limit", that is,  the singular  cohomology of $X'$ in the corresponding degrees is isomorphic to the  limit for $\eps\to 0$ of the cohomology of $s_X\ob(Z_\varepsilon)$. Indeed, if $\gi: Z' \to X'$ is an embedding topological spaces, $\gff$ is a sheaf on $X'$, then $(\gi^{*}\gff)(Z') =  \underrightarrow{\lim}_{Z_0 \supset Z', Z_0 \text{ is open in } X'}$ $\gff (Z_0)$. Now, if $K = \mathbb{C}$ 
 then $Z'$ is compact by our assumptions and $X'$ is a metric space; hence $(\gi^{*}\gff)(Z) =  \underrightarrow{\lim}$ $\gff (Z_\epsilon)$ (see Remark \ref{finiteness}(\ref{ii3})).
 
This cohomological statement is clearly weaker than its homotopical analogue in ibid. (see also the more general Theorem 2.1.3 of \cite{hammle2}). On the other hand the proof is simpler than the one in \cite{gorma} (at least, for readers that are  familiar with \'etale cohomology and do not know much about stratified Morse theory), and we succeeded in using it. 
 

\item \label{itdef}
Theorem \ref{tdef}(3) is our (very simple) substitute for \cite[\S II.5A]{gorma} 
 (that is sufficient for our applications below, and makes sense for $K$ being an arbitrary field). 
 Note that loc. cit. relies on the Thom's first isotopy lemma whose proof is really long (see \S I.1.5 of ibid.).

\item It is easily seen from the proof of Theorems \ref{tdef}(1) and \ref{tgara} that one can replace the "initial sheaf" $\znz_{X'}$ in these theorems by any other object of $S$ of $\dsn(X')$ such that $\gj'^*S$ belongs to  $\dsn(U')^{p\ge d}$. In particular, it suffices to assume that  $\gj'^*S$ is a locally constant sheaf. 

\item
Let us now prove a certain fundamental group version of Theorem \ref{tgara}(2). 

\begin{coro}\label{cpro}
  Adopt the assumptions and the notation of Theorem \ref{tgara}(2) and assume that $X'$ is connected and $\hd(U') - h \ge 1$. Then $Z'$ is connected and 
  $\pi_1(\gi')$  
is surjective; here $\pi_1$  denotes the \'etale fundamental group functor.  

\end{coro}

\begin{proof}

Recall that a variety $Y/K$ is (geometrically) connected if and only if $H_0(Y,\znz)\cong \znz$. Thus $H_0(X',\znz)\cong \znz$, Theorem \ref{tgara}(2) implies that $H_0(Z',\znz)\cong \znz$, and we obtain that $Z'$ is connected.

 Take an open normal subgroup $H''$ of $\pi_1(X', x')$, where $x' \in Z'$.
Recall 
that $H''$
corresponds to a pointed Galois covering space $(X'' , x'')$ of $(X', x')$ (see Appendix A.I.1 of \cite{frki}),
and 
 the degree of this cover equals the index of
$H''$ in $\pi_1(X', x')$. 
Consider 
 the corresponding cartesian square 
$$\begin{CD}
Z''@>{\gi''}>>  X''\\
@VV{}V
@VV{}V \\
Z'@ >{\gi'}>> X' \end{CD} $$
Now we  apply Theorem \ref{tgara}(2) for $\gi''$ instead of $\gi'$. We obtain that 
$H^0(-,\znz)(\gi'')$ is bijective; hence $Z''$ is connected as well. Thus $(Z'' , x'')\to (Z', x')$ is a pointed Galois covering space.
 Hence there is a bijection of cosets $\pi_1(Z', x')/(\pi_1(\gi')\ob (H''))\to \pi_1(X', x')/H''$ (this is equivalent to the equality of the degrees of the covering $Z'' \to Z'$ and $X'' \to X'$; see Appendix A.I of \cite{frki}).




Thus 
 the composed homomorphism $\pi_1(Z', x') \xrightarrow{\pi_1(\gi')} \pi_1(X', x') \to \pi_1(X', x') / H''$ is surjective for any open normal subgroup $H''$ of $\pi_1(X', x')$. Hence the image of $\pi_1(\gi')$ is a dense subset of $\pi_1(X', x')$. Now recall that profinite groups are compact Hausdorff spaces; thus the image of a continuous homomorphism between them is a closed subgroup. 
  Consequently, $\pi_1(\gi')$ is surjective.
  
\end{proof}

\begin{rema} \label{rpi1}
We conjecture that in the setting of Corollary \ref{cpro} if $X'$ is connected and $\hd(U') - h \ge 2$ then the homomorphism $\pi_1(\gi')$ is actually bijective; cf. Corollary 3.5 of \cite{sga212}. 
 Currently we 
  can only prove that $\Ker{\pi_1(\gi'))}$ has no continuous quotients isomorphic to $\zlz$  for any $\ell$ distinct from $p$.
    
\end{rema}

\end{enumerate}
\section{Further ideas and future plans}\label{sfut}

Now we will describe some
 more ideas related to weak Lefschetz-type theorems; 
 we plan to 
 elaborate on them in subsequent papers on this subject.

\begin{enumerate}

    
\item\label{ismooth} Clearly, 
our results can be applied to the study of 
the Picard and Brauer groups of varieties. 
Yet it may be rather difficult to control the infinitely divisible parts of these groups if it is not clear that they are zero. Thus it makes sense to apply the well-known Grothendieck's torsion statement and study the Brauer group of (the corresponding) smooth varieties.

    \item\label{iest} 
One can easily check that the goodness estimate provided by Theorem \ref{tgara}(1) is not precise (in general; we will discuss some cases in which the estimate can be improved soon). 
 In particular, it clearly makes sense to find (a "big") $w_U\in \z$ such that $s_{U*} \znz_{U'}\in \dsn(U)^{p\ge w_U}$.

 Moreover, as we have just noted, the smooth case of our results can be of special interest.
So, we are going to consider the case where the variety $U'$ is smooth. It appears that in this case one can replace the relative dimension $d_S$ in Theorem \ref{tgara} by a certain constant between $d_S$ and the 
 dimension of the generic fiber of $s_U$. 
In particular, this constant should be $0$ whenever $s_U$ is (generically finite and) {\it semismall} (see Definition III.7.3 of \cite{kw} or \S0.2.2.3 of \cite{gorma}). 

The corresponding statement 
 should be closely related to Lemma III.7.5 of \cite{kw}; it   should be a (rather general) \'etale cohomology modification of the theorem in \S II.1.1  of \cite{gorma}.  ibid. 

   \item\label{impr}
   We would also like to prove that the degree estimate 
   in Theorem \ref{tgara} can 
   be improved 
similarly to \S II.1.2 of \cite{gorma}; see \S\ref{srem}.\ref{icomp} above. Both for this purpose and for the "smooth" estimate mentioned in the previous part of this section we plan to apply a certain stratification argument. 
 An important lemma here is as follows: to establish the $w$-goodness of $s_U$ it suffices to find a stratification 
$\{U_{\al}\}$ of $U'$ such that for all the embeddings $j_{\al}$ corresponding to it we have $u_{!} (s_U\circ j_{\al})_* j_{\al}^! \znz_{U'}\in \dsn(\spe K)^{p\ge w}$.

This result may yield a generalization of Theorem 10.5(iv) of \cite{lyubez}. 

\item \label{lefl} \cite{blefl} is an alternative and less accurate version of the current paper. It contains some additional remarks related to our results.
   
      \item \label{closure} It is actually not necessary to assume that $K$ is algebraically closed and we may assume that all our schemes are reduced; see Proposition 1.4(2) of \cite{blefl}. 

\item It appears that our methods and main results can  be carried over to the complex analytic context. 
  To this end one should replace affine varieties by Stein manifolds in the formulation and  apply Theorem 5.2.16(ii) of \cite{dimca} in the proof. 
   However, the authors did not check the detail here, and do not have any really interesting applications in mind. 
\end{enumerate}

\section{A few dull proofs} \label{App}

\begin{proof}[The justification of Remark \ref{rhamm}(\ref{iexa1})]

Recall that  
in this 
 remark $N\ge 2$ and the group $\zmz$ acts on $\com^N$ as follows: $g^k ((x_1, x_2, \ldots x_N))= (\zeta_m^k \cdot x_1, \zeta_m^k \cdot x_2, \ldots \zeta_m^k \cdot x_N)$ for all $k \in \z$ and 
$(x_1, x_2, \ldots x_N)\in \com^N$ (where $g$ is a generator of  $\zmz$  and $\zeta_m$  is a primitive $m$th root of unity);   $X$ is the quotient variety $\com^N / \zmz$ by this action. By $\pi: \com^N \to X$ denote the corresponding factorization morphism. 

So let us calculate the $\znz$-cohomological depth of $X$ (for a fixed positive integer $n$).

Take  $0$ to be the origin of $\com^N$, $U=X \setminus \{\pi (0)\}$;  $\gj: U \to X$ is the corresponding open immersion, and $\gi:\spe \com\to X$ is the complementary closed immersion.

Now recall that for $c\in \z$ we have $\znz_X\in  \dsn(X)^{p\ge c}$ if and only if $\gj^!(\znz_X)\cong \gj^*(\znz_X)\in  \dsn(U)^{p\ge c}$ and $\gi^*(\znz_X)\in  \dsn(\spe \com)^{p\ge c}$; 
 see 
 (4.0.2) and Proposition 2.2.2(iii) of \cite{bbd}. 
Since $U$ is smooth, we have $\gj^*(\znz_X)\cong \znz_U\in \dsn(U)^{p\ge N}\setminus  \dsn(U)^{p\ge N+1}$ (see Proposition \ref{ppt}(\ref{ifield},\ref{itupstur})); hence $\hd(X,n)$ is equals $\min(N,\sup j:\ \gi^!\znz_X\in \dsn(\spe \com)^{p\ge j})$.

Then \S1.4.3.4 of \cite{bbd} gives a distinguished triangle $$\gi_*\gi^!\znz_X \to \znz_X \to \gj_*\gj^*\znz_X\to \gi_*\gi^!\znz_X[1].$$ 
Applying $x_*$ to this distinguished triangle (where $x: X \to \spe \com$ is the corresponding structure morphism) we obtain 
 the following distinguished triangle:
$$x_*\gi_*\gi^!\znz_X \to x_*\znz_X \to x_*\gj_*\gj^*\znz_X \to x_*\gi_*\gi^!\znz_X[1].$$

Set $u = x \circ \gj$. Clearly, $u: U \to \spe \com$ is the corresponding structure morphism and $x \circ \gi$ is $id_{\spe \com}$. Hence the latter distinguished triangle is isomorphic to $\gi^! \znz_X \to x_* \znz_X \to u_*\znz_U\to \gi^! \znz_X [1]$.
Consequently, 
 there exists a long exact sequence $$ \ldots \to H^{k-1}(U, \znz) \to H^k(\spe \com, \gi^! \znz) \to H^k(X, \znz) \to H^k(U, \znz) \to H^{k+1} (\spe \com, \gi^! \znz) \to\ldots$$

By Remark \ref{finiteness}(\ref{ii3}), $H^k(X, \znz) \cong H^k_{sing}(X, \znz)$ and $H^k(U, \znz) \cong H^k_{sing}(U, \znz)$.  Since $X(\com)$ is contractible, we have $H^k_{sing}(U, \znz) \cong H^{k+1} (\spe \com, \gi^! \znz)$ for all $k \ge 1$. 

We also obtain the following exact sequence: $$ 0 \to H^0(\spe \com, \gi^! \znz) \to H^0(X, \znz) \to H^0(U, \znz) \to H^{1} (\spe \com, \gi^! \znz) \to 0$$

Since $X$ and $U$ are connected, $H^0(X, \znz) \to H^0(U, \znz)$ is an isomorphism. Hence we conclude that $H^0(\spe \com, \gi^! \znz)= \ns$ and $H^1(\spe \com, \gi^! \znz) =\ns$.

Summing up, we proved that $H^0(\spe \com, \gi^! \znz) = H^1(\spe \com, \gi^! \znz) =\ns$ and $H^{k+1} (\spe \com, \gi^! \znz) \cong H^{k}_{sing} (U, \znz)$ for all $k \ge 1$. 

We also note that 
 $ \com^N \setminus \ns$ is simply connected; hence the map 
 $\pi': \com^N \setminus \ns \to U$ is an $m$-sheeted universal covering and $\pi_1(U) \cong \zmz$.


Now we consider two cases.

1. Suppose that $n$ and $m$ are not coprime.

   Our universal covering observation implies that $H^1_{sing}(U, \znz) \cong Hom(\pi_1(U), \znz) \neq 0$. Since $H^2 (\spe \com, \gi^! \znz) \cong H^1_{sing}(U, \znz)  \ne 0$, $\hd(X, n) = 2$. 

2. Suppose that $n$ and $m$ are coprime. 

 Since $\pi': \com^N \setminus \ns \to U$ is 
 an $m$-sheeted 
  covering,  
  we have the obvious homomorphism $\pi'^{*}: H^k_{sing}(U, \znz) \to H^k_{sing}(\com^N \setminus \ns, \znz)$ and the transfer homomorphism $\tau^*: H^k_{sing}(\com^N \setminus \ns, \znz) \to H^k_{sing}(U, \znz)$ (see \S 3.H of \cite{hatcher} for the detail). The composition $\tau^* \circ \pi'^*$ is clearly the multiplication by $m$; 
  hence it is bijective.

 Now, the singular cohomology of $\com^N\setminus\ns$ vanishes in all degrees between $1$ and $2N - 2$. Since  $\tau^* \circ \pi'^*$ factors through $\ns=H^k_{sing}(\com^N \setminus \ns)$,   
  the 
 $\znz$-cohomology of $U$ vanishes in all degrees between $1$ and $2N - 2$ as well. Therefore $H^k (\spe \com, \gi^! \znz) =  0$ for all $k < 2N$. Consequently, $\hd(X, n) = N$. 
\end{proof}

\begin{proof}[The justification of \S\ref{srem}.\ref{irlyub2}] 

First let us prove that  for $k=\com$ and any $N\ge 2$ there exists a closed immersion $t_N:S^2 \p^N\to \p^{\frac{(N+1)(N+2)}{2}-1}$.
For this purpose, note that one can combine the arguments used in the proof of \cite[Theorem 3.1.12]{Ryd} with Corollary 3.1.11(ii) of ibid. (instead of Corollary 3.1.11(i) of ibid. that was used in the proof of Theorem 3.1.12 of ibid.). In the notation of loc. cit. we take $S=\spe \com$, $d = 2$, $\mathcal{E} = \mathcal{O}_{\spe \com}^{N+1}$ (in particular, $r = N$) and obtain the existence of a closed embedding $t_N$ as wanted.

By $\Delta\cong \p^N$ denote the image of the diagonal in $(\p^N)^2$ in $S^2\p^N$.





 Now, $S^2 \p^N$ is smooth everywhere except $\Delta$. Hence 
  $S^2 \p^N = S^2U_0 \cup \ldots \cup S^2U_N \cup S$, where $S = S^2\p^N \setminus \Delta$ is smooth and $U_i\cong \af^N$ are the standard affine charts on $\p^N$. 
  Now $S$ is smooth and $U_i$ are isomorphic to each other, hence $\hd(S^2 \p^N, n) = \hd(S^2 U_0, n) = \hd(S^2 \af^N, n)$ (see Proposition \ref{hd}(\ref{itlocal}))

 Now, $f: \af^N \times \af^N \to \af^N \times \af^N$ is  $\ztz$-equivariant isomorphism, where $f(a, b) = (a + b, a - b)$, 
   the action of the generator element $g\in \ztz\setminus \ns$ on the left hand side is given by $g \cdot (a, b) = (b, a)$ and on the right hand side we set $g \cdot (a, b) = (a, - b)$. 
    Hence $f$ 
     yields an isomorphism $f': S^2 \af^N \to \af^N \times X_2^N$ of the corresponding quotiens 
 (see Remark \ref{rhamm} (\ref{iexa1}) for the definition of $X_n^N$). Thus combining Remark \ref{rhamm}(\ref{iexa1})  with Proposition \ref{hd}(\ref{igmbundle}) we obtain that $\hd(S^2 \af^N, n) = 2N$ for odd $n$ and $\hd(S^2 \af^N, n) = N + 2$ for even $n$. Consequently, 
  $\hd(S^2 \p^N, n) = 2N$ for odd $n$ and $\hd(S^2 \p^N, n) = N + 2$ for even $n$; hence the equational excess of $S^2 \p^N$ is at least $N - 2$. Moreover, set $U$ to be the maximal open subvariety of  $S^2 \p^N$  
   whose equational excess is at most $N-3$. For $N \ge 3$ let us demonstrate that $U = S^2\p^N \setminus \Delta$. 
   
   Indeed, $S^2\p^N \setminus \Delta$ is smooth; hence its equational excess equals 
   $0 \le N - 3$ . 
    Hence $S^2\p^N \setminus \Delta \subset U$. 
     Now assume that $U$ contains a closed point $v \in \Delta$. 
      Then for any point $v' \in \Delta$ 
     there exists an open neighbourhood $V'$ of $v'$ such that $V' \cong V$, where $V'$ is obtained from $V$ by an automorphism of $S^2\p^N$. Consequentially, the equational excess of $V'$ is at most $N - 3$ and $v' \in U$. Thus we obtain that $U = S^2\p^N$, that is, the equational excess of $S^2 \p^N$ is at most $N - 3$ and we get a contradiction in this case.




 



 

 \end{proof}

\begin{proof}[Proof of Lemma \ref{lgmbundle}] 
We prove the statement by induction in $w$.

 If $w<0$, then the statement is vacuous. Now assume for some $s\ge -1$ that the statement  is fulfilled for $w=s$; assume also that
  $f$ and $f'$ are $s$-connected with respect to $H_n^*$. We should verify that $H^{s}(-,\znz)(f)$ is surjective and $H^{s+1}(-,\znz)(f)$ is injective if and only if the same is true for $f'$.
  
  We recall the existence of certain Gysin long exact sequences for $G_m$-bundles.

For any 
$G_m$-bundle $b:X'\to X$ we present  $X'$ as the complement to $X$ for the line bundle $X''=X\times \af^1/G_m$ over $X$ (we consider the diagonal action of $G_m$ on the product and the zero section embedding $X\to X''$).  Corollary 1.5 of \cite{sga5e7} yields
  the existence of a 
 long exact sequence  
\begin{equation}\label{efrgys} \begin{aligned} 
\dots \to
 H^{s-1}(X',\znz)\to H^{s-2}(X,\znz)\to H^s(X,\znz) \to \\
 H^s(X',\znz)\to H^{s-1}(X,\znz) \to  H^{s+1}(X,\znz) \to
  \dots 
  \end{aligned}
\end{equation} 
and that this sequence is functorial with respect to $b$ (this long exact sequence may look somewhat strange since we have no need to keep track of Tate twists).

Now we consider these long exact sequences for $A'$ and $B'$:
\begin{equation}\resizebox{.85\hsize}{!}{$ \label{emles} 
  \begin{CD}
H^{s-1}(A',\znz)@>{}>> H^{s-2}(A,\znz)@>{}>> H^s(A,\znz) @>{}>> 
 H^s(A',\znz)@>{}>> H^{s-1}(A,\znz) @>{}>>  H^{s+1}(A,\znz) \\
@VV{H^{s-1}(-,\znz)(f')}V@VV{H^{s-2}(-,\znz)(f)}V@VV{H^{s}(-,\znz)(f)}V@VV{H^{s}(-,\znz)(f')}V@VV{H^{s-1}(-,\znz)(f)}V @VV{H^{s+1}(-,\znz)(f)}V\\
H^{s-1}(B',\znz)@>{}>> H^{s-2}(B,\znz)@>{}>> H^s(B,\znz) @>{}>> 
 H^s(B',\znz)@>{}>> H^{s-1}(B,\znz) @>{}>>  H^{s+1}(B,\znz)
\end{CD}$}\end{equation}
The five lemma yields: if $H^{s}(-,\znz)(f)$ is bijective and $H^{s+1}(-,\znz)(f)$ is injective, then $H^{s}(-,\znz)(f')$ is bijective (here we apply the lemma to the last five columns of (\ref{emles})); if  $H^{s}(-,\znz)(f')$ is bijective, then $H^{s}(-,\znz)(f)$ is bijective also  (here we apply the lemma to the first five columns of (\ref{emles})). Next, we apply the four lemma on monomorphisms to the last four columns of the diagram 
\begin{equation} \resizebox{.82\hsize}{!}{$
 \label{emles2} 
  \begin{CD}
 H^s(A',\znz)@>{}>>  H^{s-1}(A,\znz) @>{}>>  H^{s+1}(A,\znz)   @>{}>>  H^{s+1}(A',\znz)@>{}>>  H^{s}(A,\znz) \\
@VV{H^{s}(-,\znz)(f')}V@VV{H^{s-1}(-,\znz)(f)}V @VV{H^{s+1}(-,\znz)(f)}V@VV{H^{s+1}(-,\znz)(f')}V@VV{H^{s}(-,\znz)(f)}V \\
 H^s(B',\znz)@>{}>> H^{s-1}(B,\znz) @>{}>>  H^{s+1}(B,\znz)@>{}>>  H^{s+1}(B',\znz)@>{}>>  H^{s}(B,\znz) 
\end{CD} $}
\end{equation}
This
 yields the following  implication: if $H^{s+1}(-,\znz) (f)$ is injective  then $H^{s+1}(-,\znz)(f')$ is injective as well. Lastly (applying the lemma to the first four columns of  (\ref{emles2})) we obtain the following:  if $H^{s}(-,\znz)(f')$ is bijective and $H^{s+1}(-,\znz)(f')$ is injective, then $H^{s+1}(-,\znz)(f)$ is injective.

Hence we obtain the inductive assumption for $w=s+1$.
\end{proof}

\begin{proof}[Proof of Lemma \ref{product}] \label{prproduct} 
    Now 
 let us rewrite the statement of Lemma \ref{product} in terms of  
 $\dsn(\spe K)\cong D^b(\znz-\modd)$ and of the total derived functors $RH(-)$.

 By $pr$ denote $x \times id_Y$; by $C_X$ and $C_Y$ denote $x_{*}(\znz_X) $ and $ y_{*}(\znz_Y) \in D^b(\znz-\modd)$ respectively, where $y: Y \to \spe K$ is also the structure morphism.
 
 Choose some morphism $s:\pt \to X$ (that is, fix a $K$-point of $X$). Since $x\circ s=\id_{\pt}$, we obtain a splitting
 $C_X \cong \znz \bigoplus D$, where $D$ is the categorical image of the idempotent endomorphism $(\id_{C_X}- (s\circ x)_*(id_{C_X}))\in 
 D^b(\znz-\modd)(C_X, C_X)$. 
 
 Now, the Kunneth formula (see Proposition \ref{pds}(\ref{ikun})) yields that $RH(pr)$ can be described as the tensor product of the split morphism $\znz \to \znz \bigoplus D$ and $id_{C_Y}$. Hence $RH(pr): C_Y \to C_Y \otimes (\znz \bigoplus D) \cong C_Y \bigoplus (C_Y \otimes D)$ is the inclusion of the first summand.

 The cones of the morphisms $RH(x)$ and $RH(pr)$ are isomorphic to $D$ and $D \otimes C_Y$ respectively. So it suffices to prove that cohomology of $D$ is concentrated in degrees $> w$ if and only if cohomology of $D \otimes C_Y$ is concentrated in degrees $> w$. This is true since $C_Y$ has  nontrivial cohomology in degree 0 (see Proposition \ref{pds}(\ref{icohh})).
 
\end{proof}

\begin{proof}[Proof of  Lemma \ref{l-adic}] 
First we 
    prove the  $\zlmz$-part of the statement. 
    We apply induction on $n$.

    \textit{Base case.} In the case $m = 1$ the assertion 
     is given by our assumptions.

    \textit{Induction step.} We have a short exact sequence $0 \to \zlmz \to \zlmiz \to \zlz \to 0$, which gives 
     the following morphism of long exact sequences:

    \[\begin{tikzcd}[sep=tiny]
	\ldots && {H^{i} (Y, \zlz)} && {H^{i} (Y, \zlmiz)} && {H^{i} (Y, \zlmz)} && {H^{i+1} (Y, \zlz)} && \ldots \\
	\\
	\ldots && {H^{i} (X, \zlz)} && {H^{i} (X, \zlmiz)} && {H^{i} (Y, \zlmz)} && {H^{i+1} (Y, \zlz)} && \ldots
	\arrow[from=1-1, to=1-3]
	\arrow[from=1-3, to=1-5]
	\arrow[from=1-3, to=3-3]
	\arrow[from=1-5, to=1-7]
	\arrow[from=1-5, to=3-5]
	\arrow[from=1-7, to=1-9]
	\arrow[from=1-7, to=3-7]
	\arrow[from=1-9, to=1-11]
	\arrow[from=1-9, to=3-9]
	\arrow[from=3-1, to=3-3]
	\arrow[from=3-3, to=3-5]
	\arrow[from=3-5, to=3-7]
	\arrow[from=3-7, to=3-9]
	\arrow[from=3-9, to=3-11]
\end{tikzcd}\]

    By the induction hypothesis, $f$ is $w$-connected with respect to $H^{*}(-, \zlmz)$. So the five lemma and the four lemma on monomorphisms complete the induction step. 

     Now let us prove the $\z_{\ell}$-part of the assertion. 

  Note that if for some $i$ and for every $m>0$ the group $ H^i(Y, \z/\ell^m\z)$ injects into $ H^i(X, \z/\ell^m\z)$, then $H^i(Y, \z_\ell) \hookrightarrow H^i(X, \z_\ell)$,  and the obvious isomorphism version of this 
   implication statement is valid as well. Hence 
  $f$ is $w$-connected with respect to $H^{*}(-, \z_{\ell})$ indeed.

  Finally we 
   pass to the Tors-part of 
    our lemma.

    It's well-known that $H^r(Z, \z_\ell)$ is a finitely generated $\z_\ell$-module for any variety $Z/K$, and  for any $m>0$ there exists the following long exact sequence (see \S I.12 of \cite{frki}):

  \[\begin{tikzcd}[sep=tiny]
	\ldots && {H^{r-1}(Z, \mathbb{Z}/\ell^m\mathbb{Z})} && {H^r(Z, \mathbb{Z}_\ell)} && {H^r(Z, \mathbb{Z}_\ell)} && {H^r(Z, \mathbb{Z}/\ell^m\mathbb{Z})} && {H^{r+1}(Z, \mathbb{Z}_\ell)} && \ldots
	\arrow[from=1-1, to=1-3]
	\arrow[from=1-3, to=1-5]
	\arrow["{\cdot \ell^m}", from=1-5, to=1-7]
	\arrow[from=1-7, to=1-9]
	\arrow[from=1-9, to=1-11]
	\arrow[from=1-11, to=1-13]
\end{tikzcd}\]


  Now let us check that $\mathrm{Tors}(H^{w}(-, \z_{\ell})(f))$ is bijective. Since $f$ is $w$-connected with respect to  $H^{*}(-, \z_{\ell})$ and also with respect to $H_{\ell^m}^*$ for any $m>0$,  the long exact sequence above gives the following  diagram (whose rows are exact):

\[\begin{tikzcd}[sep=tiny]
	\ldots && {H^{w-1}(Y, \mathbb{Z}_\ell)} && {H^{w-1}(Y, \mathbb{Z}/\ell^m\mathbb{Z})} && {H^{w}(Y, \mathbb{Z}_\ell)} && {H^{w}(Y, \mathbb{Z}_\ell)} && \ldots && {} \\
	\\
	\\
	\\
	\ldots && {H^{w-1}(X, \mathbb{Z}_\ell)} && {H^{w-1}(X, \mathbb{Z}/\ell^m\mathbb{Z})} && {H^{w}(X, \mathbb{Z}_\ell)} && {H^{w}(X, \mathbb{Z}_\ell)} && \ldots
	\arrow[from=1-1, to=1-3]
	\arrow[from=1-3, to=1-5]
	\arrow["{ \rotatebox[origin=c]{90}{$\sim$}}", from=1-3, to=5-3]
	\arrow[from=1-5, to=1-7]
	\arrow["{ \rotatebox[origin=c]{90}{$\sim$}}", from=1-5, to=5-5]
	\arrow["{\cdot \ell^m}", from=1-7, to=1-9]
	\arrow[hook', from=1-7, to=5-7]
	\arrow[from=1-9, to=1-11]
	\arrow[hook', from=1-9, to=5-9]
	\arrow[from=5-1, to=5-3]
	\arrow[from=5-3, to=5-5]
	\arrow[from=5-5, to=5-7]
	\arrow["{\cdot \ell^m}", from=5-7, to=5-9]
	\arrow[from=5-9, to=5-11]
 \end{tikzcd}\]

  Since $H^w(X, \mathbb{Z}_\ell)$ and $H^w(Y, \mathbb{Z}_\ell)$ are finitely generated $\z_\ell$-modules (see Remark \ref{finiteness}(\ref{ii1})), there exists $N>0$  such  that $ \ell^{N}\mathrm{Tors}(H^w(X, \mathbb{Z}_\ell)) = \ell^{N}\mathrm{Tors}(H^w(Y, \mathbb{Z}_\ell)) = \ns$. Thus we have the following commutative diagrams whose rows are exact:

  \[\begin{tikzcd}[sep=small]
	{H^{w-1}(Y, \mathbb{Z}_\ell)} && {H^{w-1}(Y, \mathbb{Z}/\ell^N\mathbb{Z})} && {\mathrm{Tors}H^w(Y, \mathbb{Z}_\ell)} && 0 \\
	\\
	{H^{w-1}(X, \mathbb{Z}_\ell)} && {H^{w-1}(X, \mathbb{Z}/\ell^N\mathbb{Z})} && {\mathrm{Tors}H^{w}(X, \mathbb{Z}_\ell)} && 0
	\arrow[from=1-1, to=1-3]
	\arrow["{ \rotatebox[origin=c]{90}{$\sim$}}", from=1-1, to=3-1]
	\arrow[from=1-3, to=1-5]
	\arrow["{ \rotatebox[origin=c]{90}{$\sim$}}", from=1-3, to=3-3]
	\arrow[from=1-5, to=1-7]
	\arrow[hook', from=1-5, to=3-5]
	\arrow["{ \rotatebox[origin=c]{90}{$\sim$}}", from=1-7, to=3-7]
	\arrow[from=3-1, to=3-3]
	\arrow[from=3-3, to=3-5]
	\arrow[from=3-5, to=3-7]
\end{tikzcd}\]

    Hence applying the four lemma on epimorphisms we obtain $\mathrm{Tors}(H^w(Y, \mathbb{Z}_\ell)) \twoheadrightarrow \mathrm{Tors}(H^w(X, \mathbb{Z}_\ell))$; thus $\mathrm{Tors}(H^w(-, \mathbb{Z}_\ell)(f))$ is bijective.

  It remains to verify the  injectivity of $\mathrm{Tors}(H^{w+1}(-, \z_{\ell})(f))$
   We obtain the following commutative diagram whose rows are exact:

   \[\begin{tikzcd}[sep=tiny]
	\ldots && {H^{w}(Y, \mathbb{Z}_\ell)} && {H^{w}(Y, \mathbb{Z}/\ell^m\mathbb{Z})} && {H^{w+1}(Y, \mathbb{Z}_\ell)} && {H^{w+1}(Y, \mathbb{Z}_\ell)} && \ldots && {} \\
	\\
	\\
	\\
	\ldots && {H^{w}(X, \mathbb{Z}_\ell)} && {H^{w}(X, \mathbb{Z}/\ell^m\mathbb{Z})} && {H^{w+1}(X, \mathbb{Z}_\ell)} && {H^{w+1}(X, \mathbb{Z}_\ell)} && \ldots
	\arrow[from=1-1, to=1-3]
	\arrow[from=1-3, to=1-5]
	\arrow[hook', from=1-3, to=5-3]
	\arrow[from=1-5, to=1-7]
	\arrow[hook', from=1-5, to=5-5]
	\arrow["{\cdot \ell^m}", from=1-7, to=1-9]
	\arrow[from=1-7, to=5-7]
	\arrow[from=1-9, to=1-11]
	\arrow[from=1-9, to=5-9]
	\arrow[from=5-1, to=5-3]
	\arrow[from=5-3, to=5-5]
	\arrow[from=5-5, to=5-7]
	\arrow["{\cdot \ell^m}", from=5-7, to=5-9]
	\arrow[from=5-9, to=5-11]
 \end{tikzcd}\]

   Since $H^{w+1}(X, \mathbb{Z}_\ell)$ and $H^{w+1}(Y, \mathbb{Z}_\ell)$ are finitely generated $\z_\ell$-modules (see Remark \ref{finiteness}(\ref{ii1})), there exists $N>0$  such  that $ \ell^{N}\mathrm{Tors}(H^{w+1}(X, \mathbb{Z}_\ell)) = \ell^{N}\mathrm{Tors}(H^{w+1}(Y, \mathbb{Z}_\ell)) = \ns$. Thus we have the following commutative diagrams whose rows are exact:

    \[\begin{tikzcd}[sep=small]
	{\mathrm{Tors}H^{w}(Y, \mathbb{Z}_\ell)} && {H^{w}(Y, \mathbb{Z}/\ell^N\mathbb{Z})} && {\mathrm{Tors}H^{w+1}(Y, \mathbb{Z}_\ell)} && 0 \\
	\\
	{\mathrm{Tors}H^{w}(X, \mathbb{Z}_\ell)} && {H^{w}(X, \mathbb{Z}/\ell^N\mathbb{Z})} && {\mathrm{Tors}H^{w+1}(X, \mathbb{Z}_\ell)} && 0
	\arrow[from=1-1, to=1-3]
	\arrow["{ \rotatebox[origin=c]{90}{$\sim$}}", from=1-1, to=3-1]
	\arrow[from=1-3, to=1-5]
	\arrow[hook', from=1-3, to=3-3]
	\arrow[from=1-5, to=1-7]
	\arrow[from=1-5, to=3-5]
	\arrow["{ \rotatebox[origin=c]{90}{$\sim$}}",, from=1-7, to=3-7]
	\arrow[from=3-1, to=3-3]
	\arrow[from=3-3, to=3-5]
	\arrow[from=3-5, to=3-7]
\end{tikzcd}\]

Hence applying the four lemma on monomorphisms we obtain $\mathrm{Tors}(H^{w+1}(Y, \mathbb{Z}_\ell)) \hookrightarrow \mathrm{Tors}(H^{w+1}(X, \mathbb{Z}_\ell))$; thus $\mathrm{Tors}(H^{w+1}(-, \mathbb{Z}_\ell)(f))$ is injective.

\end{proof}

\end{document}